\documentclass[article]{siamart171218}

\usepackage{lipsum}
\usepackage{amsfonts}
\usepackage{graphicx}
\usepackage{epstopdf}
\usepackage{algorithmic}
\ifpdf
  \DeclareGraphicsExtensions{.eps,.pdf,.png,.jpg}
\else
  \DeclareGraphicsExtensions{.eps}
\fi

\newcommand{\TheTitle}{The Convex Feasible Set Algorithm \\for Real Time Optimization in Motion Planning} 
\newcommand{\TheAuthors}{C. Liu, C. Lin and M. Tomizuka}

\headers{The Convex Feasible Set Algorithm}{\TheAuthors}

\title{{\TheTitle}}%\thanks{This work was funded by Fanuc}}

\author{
  Changliu Liu\thanks{Department of Mechanical Engineering, University of California at Berkeley, Berkeley, CA 94720, USA
    (\email{changliuliu@berkeley.edu}, \email{chung\_yen@berkeley.edu}, \email{tomizuka@berkeley.edu}). The research was supported by FANUC Corporation.}
  \and
  Chung-Yen Lin\footnotemark[2]
  \and
  Masayoshi Tomizuka\footnotemark[2]
}

\usepackage{amsopn}

\ifpdf
\hypersetup{
  pdftitle={\TheTitle},
  pdfauthor={\TheAuthors}
}
\fi

\usepackage{graphicx}
\usepackage{subfig}
\usepackage{amsmath, amssymb}
\usepackage{multirow}

\newsiamthm{remark}{Remark}
\newsiamthm{assump}{Assumption}

\begin{document}

\maketitle

\begin{abstract}
With the development of robotics, there are growing needs for real time motion planning. However, due to obstacles in the environment, the planning problem is highly non-convex, which makes it difficult to achieve real time computation using existing non-convex optimization algorithms. This paper introduces the convex feasible set algorithm (CFS) which is a fast algorithm for non-convex optimization problems that have convex costs and non-convex constraints. The idea is to find a convex feasible set for the original problem and iteratively solve a sequence of subproblems using the convex constraints. The feasibility and the convergence of the proposed algorithm are proved in the paper. The application of this method on motion planning for mobile robots is discussed. The simulations demonstrate the effectiveness of the proposed algorithm.
\end{abstract}

% REQUIRED
\begin{keywords}
  Non-convex optimization, non-differentiable optimization, robot motion planning
\end{keywords}

% REQUIRED
\begin{AMS}
  90C55, 90C26, 68T40, 93C85
\end{AMS}

\section{Introduction}

Although great progresses have been made in robot motion planning \cite{latombe2012robot}, the field is still open for research regarding real time planning in dynamic uncertain environment. The applications include but are not limited to real time navigation \cite{frazzoli2002real}, autonomous driving \cite{kuwata2009real, liu2016enabling}, robot arm manipulation and human robot cooperation \cite{liu2016algorithmic}. To achieve safety and efficiency, robot motion should be re-planned from time to time when new information is obtained during operation. The motion planning algorithm should run fast enough to meet the real time requirement for planning and re-planning.

In this paper, we focus on optimization-based motion planning methods, which fit into the framework of model-predictive control (MPC) \cite{howard2010receding}, where the optimal trajectory is obtained by solving a constrained optimization at each time step. As there are obstacles in the environment, the constraints are highly non-convex, which makes the problem hard to solve in real time. Various methods have been developed to deal with the non-convexity \cite{nocedal2006numerical, strongin2013global}. One popular way is through convexification \cite{tawarmalani2002convexification}, e.g., transforming the non-convex problem into a convex one. Some authors propose to transform the non-convex problem to semidefinite programming (SDP) \cite{eichfelder2013set}. Some authors introduce lossless convexification by augmenting the space \cite{accikmecse2013lossless, harris2014lossless}. And some authors use successive linear approximation to remove non-convex constraints \cite{liu2013autonomous, liu2014solving}. However, the first method can only handle quadratic cost functions. The second approach highly depends on the linearity of the system and may not be able to handle diverse obstacles. The third approach may not generalize to non-differentiable problems. Another widely-used convexification method is the sequential quadratic programming (SQP) \cite{spellucci1998new,tone1983revisions}, which approximates the non-convex problem as a sequence of quadratic programming (QP) problems and solves them iteratively. The method has been successfully applied to offline robot motion planning \cite{johansen2004constrained, schulman2013finding}. However, as SQP is a generic algorithm, the unique geometric structure of the motion planning problems is neglected, which usually results in failure to meet the real time requirement in engineering applications. 

Typically, in a motion planning problem, the constraints are physically determined, while the objective function is designed to be convex \cite{van2011lqg, ratliff2009chomp}. The non-convexity mainly comes from the physical constraints. Regarding this observation, a fast algorithm called the convex feasible set algorithm (CFS) is proposed in this paper to solve optimization-based motion planning problems with convex objective functions and non-convex constraints. 

The main idea of the CFS algorithm is to transform the original problem into a sequence of convex subproblems by obtaining convex feasible sets within the non-convex domain, then iteratively solve the convex subproblems until convergence. The idea is similar to SQP in that it tries to solve several convex subproblems iteratively. The difference between CFS and SQP lies in the way to obtain the convex subproblems. The geometric structure of the original problem is fully considered in CFS. This strategy will make the computation faster than conventional SQP and other non-convex optimization methods such as interior point (ITP) \cite{vanderbei1999interior}, as will be demonstrated later. Moreover, local optima is guaranteed.

It is worth noting that the convex feasible set in the trajectory space can be regarded as a convex corridor. The idea of using convex corridors to simplify the motion planning problems has been discussed in \cite{chen2015path, zhu2015convex}. However, these methods are application-specific without theoretical guarantees. 
In this paper, we consider general non-convex and non-differentiable optimization problems (which may not only arise from motion planning problems, but also other problems) and provide theoretical guarantees of the method. 

The remainder of the paper is organized as follows. \Cref{sec: benchmark problem} proposes a benchmark optimization problem. \Cref{sec: solve optimization} discusses the proposed CFS algorithm in solving the benchmark problem. \Cref{sec: proofs} shows the feasibility and convergence of the algorithm. \Cref{sec: mobile robot} illustrates the application of the algorithm on motion planning problems for mobile robots. \Cref{sec: conclusion} concludes the paper.

\section{The Optimization Problem\label{sec: benchmark problem}}
\subsection{The Benchmark Problem\label{sec: benchmark}}
Consider an optimization problem with a convex cost function but non-convex constraints, i.e.,
\begin{equation}
\min_{\mathbf{x}\in\Gamma} J(\mathbf{x}),\label{eq: the problem}
\end{equation}
where $\mathbf{x}\in\mathbb{R}^n$ is the decision variable and the problem follows two assumptions.
\begin{assump} [Cost]
$J:\mathbb{R}^n\rightarrow \mathbb{R}^+$ is smooth and strictly convex. \label{assump: cost}
\end{assump}
\begin{assump}[Constraint]
The set $\Gamma\subset\mathbb{R}^n$ is connected and closed, with piecewise smooth and non-self-intersecting boundary $\partial\Gamma$. For every point $x\in\Gamma$, there exists an $n$-dimensional convex polytope $P\subset\Gamma$ such that $x\in P$.\label{assump: constraint}
\end{assump}

\cref{assump: cost} implies that $J$ is radially unbounded, i.e., $J(\mathbf{x})\rightarrow\infty$ when $\|\mathbf{x}\|\rightarrow\infty$. 
\cref{assump: constraint} specifies the geometric features of the feasible set $\Gamma$, where the first part deals with the topological features of $\Gamma$ and the second part ensures that there is a convex neighborhood for any point in $\Gamma$. Note that equality constraints in $\Gamma$ are excluded by \cref{assump: constraint} as the dimension of the neighborhood for any point satisfying a equality constraint is strictly less than $n$.

The geometric structure of problem \cref{eq: the problem} is illustrated in \cref{fig: cfs illustration}. The contour represents the cost function $J$, while the gray parts represent $\Gamma^c$. There are two disjoint components in $\Gamma^c$. The goal is to find a local optimum (hopefully global optimum) starting from the initial reference point (blue dot). As shown in \cref{fig: cfs illustration}, the problem is highly non-convex and the non-convexity comes from the constraints. To make the computation efficient, we propose the convex feasible set algorithm in this paper, which transforms problem \cref{eq: the problem} into a sequence of convex optimizations by obtaining a sequence of convex feasible sets inside the non-convex domain $\Gamma$. As shown in \cref{fig: cfs}, the idea is implemented iteratively. At current iteration, a convex feasible set for the current reference point (blue dot) is obtained. The optimal solution in the convex feasible set (black dot) is set as the reference point for the next iteration. 
The formal mathematical description of this algorithm will be discussed in \cref{sec: solve optimization}. The feasibility of this method, i.e., the existence of an $n$-dimensional convex feasible set, is implied by \cref{assump: constraint}.  
Nonetheless, in order to compute the convex feasible set efficiently, we still need an analytical description of the constraint, which will be discussed in \cref{sec: analytical constraint}.

\begin{figure}[t]
\begin{center}
\subfloat[Iteration 1.\label{fig: cfs illustration}]{
\includegraphics[width=4.2cm]{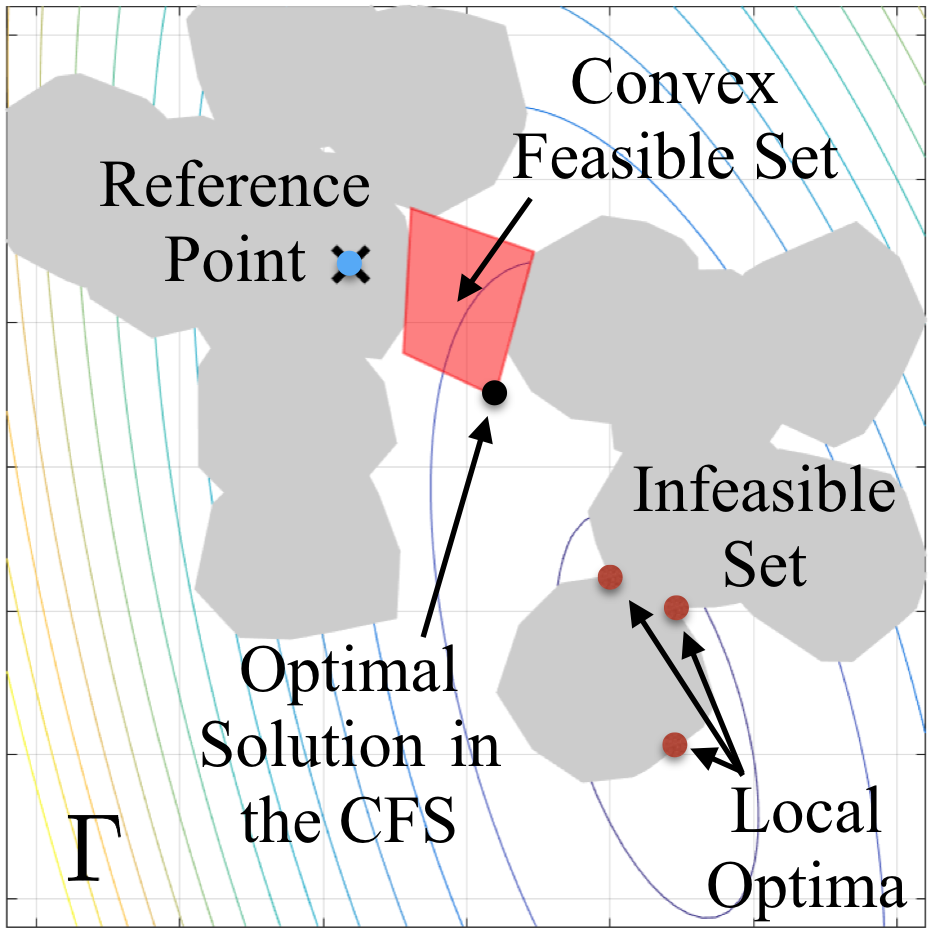}}
\subfloat[Iteration 2.\label{fig: cfs illustration 2}]{
\includegraphics[width=4.2cm]{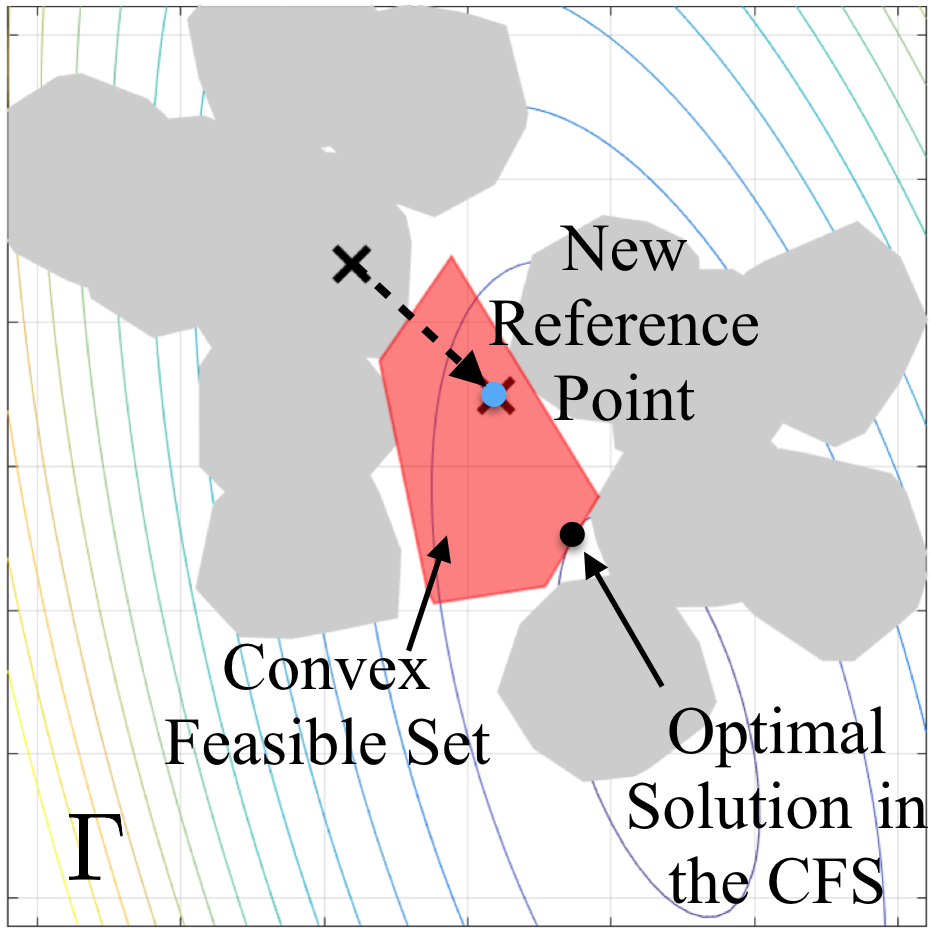}}
\subfloat[Iteration 3.\label{fig: cfs illustration 3}]{
\includegraphics[width=4.2cm]{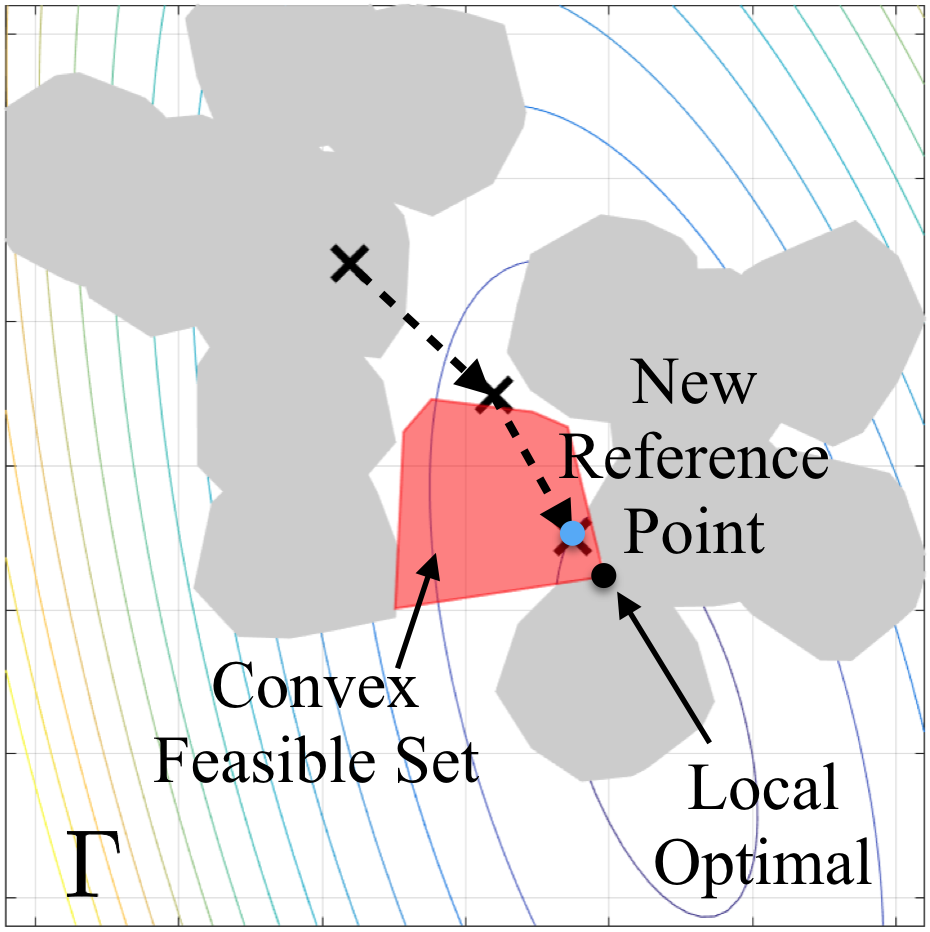}}
\caption{Geometry of problem \cref{eq: the problem} and the idea of the convex feasible set algorithm.}
\label{fig: cfs}
\end{center}
\end{figure}

\subsection{Analytical Representation of the Constraints\label{sec: analytical constraint}}
The set $\Gamma$ will be represented analytically by several inequality constraints. 
It is called a semi-convex decomposition of $\Gamma$ if 
\begin{equation}
\Gamma=\bigcap_{i=1}^N\{\mathbf{x}:\phi_i(\mathbf{x})\geq 0\}=\bigcap_{i=1}^N\Gamma_i,\label{eq: gamma}
\end{equation}
for $N$ continuous, piecewise smooth and semi-convex functions $\phi_i:\mathbb{R}^n\rightarrow\mathbb{R}$ such that $\Gamma_i:=\{\mathbf{x}:\phi_i(\mathbf{x})\geq 0\}$ with $\partial \Gamma_i = \{\mathbf{x}:\phi_i(\mathbf{x})= 0\}$. Semi-convexity \cite{Colesanti2000} of $\phi_i$ implies that there exists a positive semi-definite $H_i^*\in\mathbb{R}^{n\times n}$ such that the function 
\begin{equation}
\tilde {\phi}_i(\mathbf{x}) := \phi_i(\mathbf{x}) + \frac{1}{2}(\mathbf{x}-\mathbf{x}_0)^TH_i^* (\mathbf{x}-\mathbf{x}_0),\label{eq: semi-convex defn}
\end{equation}
is convex in $\mathbf{x}$ for any $\mathbf{x}_0\in\mathbb{R}^n$. Or in other words, the hessian of $\phi_i$ is bounded below. Note that $\Gamma_i^c$'s are not required to be disjoint and $N$ should be greater than or equal to the number of disjoint components in $\Gamma^c$. The decomposition from $\Gamma$ to $\Gamma_i$'s is not unique. Neither is the function $\phi_i$ that represents $\Gamma_i$. In many cases, $\phi_i$ can be chosen as a signed distance function to $\partial\Gamma_i$, which will be discussed in \cref{sec: transform the problem}. 

Before introducing more conditions on the decomposition \cref{eq: gamma}, analytical properties of the functions $\phi_i$'s will be studied first. Since $\tilde {\phi}_i$ is convex, then for any $\mathbf{x}_0,  v\in\mathbb{R}^n$, $\tilde {\phi}_i(\mathbf{x}_0+v)-2\tilde {\phi}_i(\mathbf{x}_0)+\tilde {\phi}_i(\mathbf{x}_0-v)\geq 0$. Consider \cref{eq: semi-convex defn}, the following inequality holds for any semi-convex functions,
\begin{equation}
\phi_i(\mathbf{x}_0+v)-2\phi_i(\mathbf{x}_0)+\phi_i(\mathbf{x}_0-v)\geq -v^TH_i^* v.\label{eq: semi-convexity}
\end{equation}
Moreover, since convex functions are locally Lipschitz \cite{convexfunctionlipschitz}, $\phi_i$ is also locally Lipschitz as implied by definition \cref{eq: semi-convex defn}. However, as $\phi_i$ is only piecewise smooth, it may not be differentiable everywhere. For any $v\in\mathbb{R}^n$, define the one-side directional derivative $\partial_v$\footnote{Note that $\partial$ refers boundary when followed by a set, e.g.,  $\partial \Gamma$. It means derivative when followed by a function, e.g.,  $\partial_v\phi_i$.} as
\begin{equation}
\partial_v\phi_i(\mathbf{x}) := \lim_{a\rightarrow 0^+} \frac{\phi_i(\mathbf{x}+av)-\phi_i(\mathbf{x})}{a}.\label{eq: def directional derivative}
\end{equation}
For any $\|v\|=1$, $\partial_{v}\phi_i(\mathbf{x})$ is bounded locally since $\phi_i$ is locally Lipschitz. 
If $\phi_i$ is smooth at direction $v$ at point $\mathbf{x}$, then
\begin{equation}
\lim_{a\rightarrow 0^+} \frac{\phi_i(\mathbf{x}+av)-\phi_i(\mathbf{x})}{a} = \lim_{a\rightarrow 0^-} \frac{\phi_i(\mathbf{x}+av)-\phi_i(\mathbf{x})}{a} = \lim_{a\rightarrow 0^+} \frac{\phi_i(\mathbf{x}-av)-\phi_i(\mathbf{x})}{-a},\label{eq: smooth directional derivative derivation}
\end{equation}
where the second equality is by taking negative of $a$. By definition \cref{eq: def directional derivative}, the right-hand side of \cref{eq: smooth directional derivative derivation} equals to $-\partial_{-v}\phi_i(\mathbf{x})$, which implies that $\partial_v\phi_i(\mathbf{x}) = -\partial_{-v}\phi_i(\mathbf{x})$. Let $\mathcal{S}(\phi_i,\mathbf{x}):=\{v\in\mathbb{R}^n: \partial_v\phi_i(\mathbf{x}) + \partial_{-v}\phi_i(\mathbf{x})=0\}$ denote all the smooth directions of function $\phi_i$ at point $\mathbf{x}$. The directional derivatives satisfy the following properties.

\begin{lemma}[Properties of Directional Derivatives]\label{lemma: inequalities of directional derivative}
If $\phi_i$ is continuous, piecewise smooth, and semi-convex, then for any $\mathbf{x},v,v_1,v_2\in\mathbb{R}^n$ and $b\in\mathbb{R}$ such that $v=v_1+v_2$, the following inequalities hold,
\begin{align}
0&\leq\partial_v\phi_i(\mathbf{x}) + \partial_{-v}\phi_i(\mathbf{x}),\label{eq: directional derivative}\\
b\partial_v\phi_i(\mathbf{x}) &\leq \partial_{bv}\phi_i(\mathbf{x}),\label{eq: multiplicity}\\
\partial_{v}\phi_i(\mathbf{x}) &\leq \partial_{v_1}\phi_i(\mathbf{x}) + \partial_{v_2}\phi_i(\mathbf{x}).\label{eq: directional derivative vector space}\end{align}
The equalities in \cref{eq: directional derivative} and \cref{eq: multiplicity} are achieved when $v\in\mathcal{S}(\phi_i,\mathbf{x})$. The equality in \cref{eq: directional derivative vector space} is achieved when $v_1,v_2\in\mathcal{S}(\phi_i,\mathbf{x})$.  
\end{lemma}
\begin{proof}
If $\phi_i$ is semi-convex, \cref{eq: semi-convexity} implies that for any scalar $a$ and vector $v$,
\begin{equation*}
\frac{\phi_i(\mathbf{x}+av)-2\phi_i(\mathbf{x})+\phi_i(\mathbf{x}-av)}{a}\geq -av^TH_i^* v.
\end{equation*}
Let $a\rightarrow 0^+$. The left-hand side approaches $\partial_v\phi_i(\mathbf{x}) + \partial_{-v}\phi_i(\mathbf{x})$, while the right-hand side approaches $0$ in the limits. Hence \cref{eq: directional derivative} holds. By definition \cref{eq: def directional derivative}, $\partial_{bv}\phi_i(\mathbf{x}) = b\partial_{v}\phi_i(\mathbf{x})$ when $b \geq 0$. When $b<0$, by \cref{eq: directional derivative}, $-|b|\partial_v\phi_i(\mathbf{x}) \leq |b|\partial_{-v}\phi_i(\mathbf{x}) =\partial_{-|b|v}\phi_i(\mathbf{x}) = \partial_{bv}\phi_i(\mathbf{x})$. Hence \cref{eq: multiplicity} holds. The equality holds when $\partial_v\phi_i(\mathbf{x}) + \partial_{-v}\phi_i(\mathbf{x})=0$,  i.e., $v\in\mathcal{S}(\phi_i,\mathbf{x})$. Moreover, \cref{eq: semi-convexity} also implies
\begin{align*}
&-\frac{a^2}{4}(v_1-v_2)^TH_i^* (v_1-v_2)\leq \phi_i(\mathbf{x}+av_1) - 2\phi_i(\mathbf{x}+a\frac{v}{2}) + \phi_i(\mathbf{x}+av_2)\\
&= \phi_i(\mathbf{x}+av_1) - \phi_i(\mathbf{x})  + \phi_i(\mathbf{x}+av_2) - \phi_i(\mathbf{x}) - 2[\phi_i(\mathbf{x}+a\frac{v}{2}) - \phi_i(\mathbf{x})].
\end{align*}
Divide the both sides by $a$ and take $a\rightarrow 0^+$. Then the left-hand side approaches $0$, while the right-hand side approaches $\partial_{v_1}\phi_i(\mathbf{x}) + \partial_{v_2}\phi_i(\mathbf{x}) - \partial_{v}\phi_i(\mathbf{x})$. Hence \cref{eq: directional derivative vector space} holds. When $v_1,v_2\in\mathcal{S}(\phi_i,\mathbf{x})$, 
\begin{equation}
0\leq \partial_v\phi_i(\mathbf{x}) + \partial_{-v}\phi_i(\mathbf{x}) \leq \partial_{v_1}\phi_i(\mathbf{x}) + \partial_{v_2}\phi_i(\mathbf{x}) + \partial_{-v_1}\phi_i(\mathbf{x}) + \partial_{-v_2}\phi_i(\mathbf{x}) = 0.\label{eq: equality vector space}
\end{equation}
The first inequality is due to \cref{eq: directional derivative}; the second inequality is due to \cref{eq: directional derivative vector space}. Hence the equality in \cref{eq: directional derivative vector space} is attained.
\end{proof}

Define the sub-differential of $\phi_i$ at $\mathbf{x}$ as
\begin{equation}
D\phi_i(\mathbf{x}):=\{d\in\mathbb{R}^n : d\cdot v \leq \partial_v\phi_i(\mathbf{x}),\forall v\in\mathbb{R}^n\}.\label{eq: sub-differential}
\end{equation}
The validity of the definition, i.e., the right-hand side of \cref{eq: sub-differential} is non empty, can be verified by \cref{lemma: inequalities of directional derivative}. By \cref{eq: multiplicity} and \cref{eq: directional derivative vector space}, for any $v_1,v_2\in\mathcal{S}(\phi_i,\mathbf{x})$, $a,b\in\mathbb{R}$ and $v = av_1+bv_2$, we have $\partial_v\phi_i(\mathbf{x}) = a\partial_{v_1}\phi_i(\mathbf{x}) + b\partial_{v_2}\phi_i(\mathbf{x})$ and $\partial_v\phi_i(\mathbf{x}) + \partial_{-v}\phi_i(\mathbf{x}) = 0$. Hence $v\in \mathcal{S}(\phi_i,\mathbf{x})$. We can conclude that 1) $\mathcal{S}(\phi_i,\mathbf{x})$ is a linear subspace of $\mathbb{R}^n$ and 2) the function induced by the directional derivative $v\mapsto\partial_v\phi_i(\mathbf{x})$ is a sub-linear function\footnote{A function $f$ is called sub-linear if it satisfies positive homogeneity $f(ax)=af(x)$ for $a>0$, and sub-additivity $f(x+y)\leq f(x)+f(y)$.} on $\mathbb{R}^n$ and a linear function on $\mathcal{S}(\phi_i,\mathbf{x})$. By Hahn-Banach Theorem \cite{folland2013real}, there exists a vector $d\in \mathbb{R}^n$ such that $d\cdot v = \partial_v\phi_i(\mathbf{x})$ for  $v\in \mathcal{S}(\phi_i,\mathbf{x})$ and $ d\cdot v\leq \partial_v\phi_i(\mathbf{x})$ for $v\in\mathbb{R}^n$. Moreover, as the unit directional derivative is bounded, the sub-gradients are also bounded. Hence the definition in \cref{eq: sub-differential} is justified. 
The elements in $D\phi_i(\mathbf{x})$ are called sub-gradients. 
When $\phi_i$ is smooth at $\mathbf{x}$, $D\phi_i(\mathbf{x})$ reduces to a singleton set which contains only the gradient $\nabla\phi_i(\mathbf{x})$ such that $\nabla\phi_i(\mathbf{x})\cdot v = \partial_v\phi_i(\mathbf{x})$ for all $v\in\mathbb{R}^n$. The definition \cref{eq: sub-differential} follows from Clarke (generalized) sub-gradients for non-convex functions \cite{clarke1975generalized}. 

With the definition of sub-differential for continuous, piecewise smooth and semi-convex functions in \cref{eq: sub-differential}, the following assumption regarding the analytical representation is made.
\begin{assump}[Analytical Representations]\label{assump: regularity}
For $\Gamma$ satisfying \cref{assump: constraint}, there exists a semi-convex decomposition \cref{eq: gamma} such that 1) $D\phi_i(\mathbf{x})\neq\{0\}$ for all $\mathbf{x}$, 2) $0\notin D\phi_i(\mathbf{x})$ if $\mathbf{x}\in\partial\Gamma_i$, and
3)  for any $\mathbf{x}$ such that $I:=\{i:\phi_i(\mathbf{x})=0\}\neq \emptyset$, there exists $v\in\mathbb{R}^n$ such that $\partial_v\phi_i(\mathbf{x}) < 0$ for all $i\in I$.
\end{assump}

Note that the hypothesis that any $\Gamma$ that satisfies \cref{assump: constraint} has a semi-convex decomposition that satisfies \cref{assump: regularity} will be verified in our future work. A method to construct the desired $\phi_i$'s is discussed in \cite{liu2017convex}. 

The first condition in \cref{assump: regularity} ensures that $\phi_i$ will not have smooth extreme points. 
Geometrically, the second condition in \cref{assump: regularity} implies that there cannot be any concave corners\footnote{Some authors name convex corners as outer corners and concave corners as inner corners \cite{Matveev20131268}.} 
in $\Gamma_i^c$ or convex corners in $\Gamma_i$. Suppose $\Gamma_i^c$ has a concave corner at $\mathbf{x}\in\partial\Gamma_i$. Since $0\notin D\phi_i(\mathbf{x})$, we can choose a unit vector $v$ such that $\partial_v\phi_i(\mathbf{x}) <0$ and $\partial_{-v}\phi_i(\mathbf{x}) <0$ as shown in \cref{fig: concave corner}. Then \cref{eq: directional derivative} is violated, which contradicts with the assumption on semi-convexity. 
Nonetheless, concave corners are allowed in $\Gamma^c$, but should only be formulated by a union of several intersecting $\Gamma_i^c$'s as shown in \cref{fig: convex corner}.  
In the example, the set $\Gamma = \{\mathbf{x} = (x_1,x_2): \min(|x_1|-1,|x_2|-1) \geq 0\}$ is partitioned into two sets $\Gamma_1 = \{\mathbf{x}: |x_1|-1\geq 0\}$ and $\Gamma_2 = \{\mathbf{x}: |x_2|-1\geq 0\}$. Both $\phi_1 =  |x_1|-1$ and $\phi_2 =  |x_1|-1$ satisfy \cref{assump: regularity}. Without the partition, $\phi = \min(|x_1|-1,|x_2|-1)$ violates the condition on semi-convexity\footnote{Let $\mathbf{x} = (1,1)$ and $v = (\cos \frac{\pi}{4},\sin\frac{\pi}{4})$. Then $\phi_i(\mathbf{x}+av)-2\phi_i(\mathbf{x})+\phi_i(\mathbf{x}-av) = - a\cos\frac{\pi}{4} -  a\sin\frac{\pi}{4} = -\sqrt{2}a$, which can not be greater than any $-a^2 v^TH_i^* v$ when $a$ is small.}. 
The third condition in \cref{assump: regularity} implies that once $\Gamma_i^c$'s intersect, they should have common interior among one another. For example, the decomposition in \cref{fig: convex corner invalid} is not allowed. In this case, the obstacle is partitioned into five components. $\partial\Gamma_2^c$, $\partial\Gamma_3^c$, and $\partial\Gamma_5^c$ intersect at $\mathbf{x}$. However, there does not exist $v\in\mathbb{R}^n$ such that $\partial_v\phi_i(\mathbf{x}) < 0$ for all $i\in \{2,3,5\}$ as the interiors of $\Gamma_2^c$ and $\Gamma_3^c$ do not intersect. This condition is enforced in order to ensure that the computed convex feasible set is non empty as will be discussed in \cref{lemma: existence}.

In this following discussion, $\phi_i$'s and $\Gamma_i$'s are referred as the semi-convex decomposition of $\Gamma$ that satisfies \cref{assump: regularity}.

\begin{figure}[t]
\begin{center}
\subfloat[The constraint $\Gamma$.]{\label{fig: concave corner}\includegraphics[width=2.8cm]{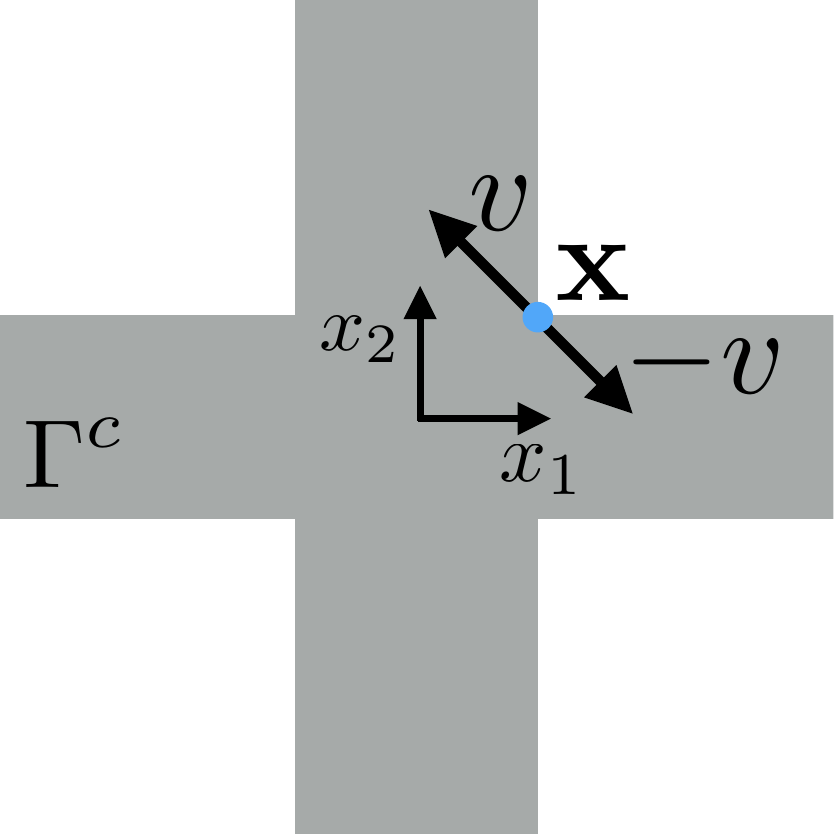}}
\subfloat[Valid decomposition of $\Gamma$.]{\label{fig: convex corner}\makebox[4cm][c]{\includegraphics[width=2.8cm]{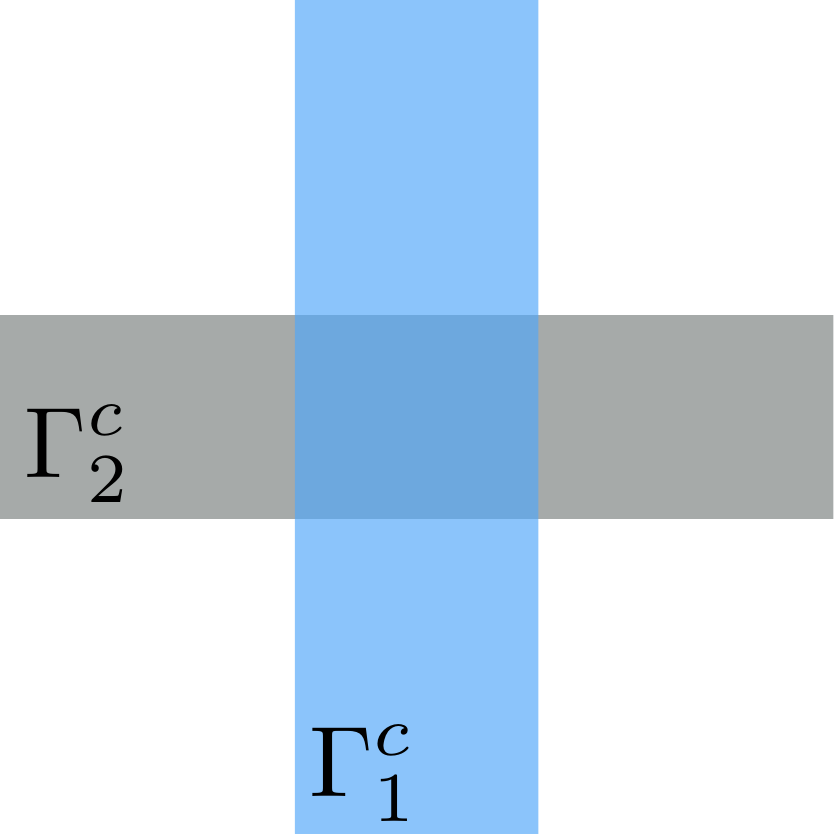}}}
\subfloat[Invalid decomposition of $\Gamma$.]{\label{fig: convex corner invalid}\makebox[4cm][c]{\includegraphics[width=2.8cm]{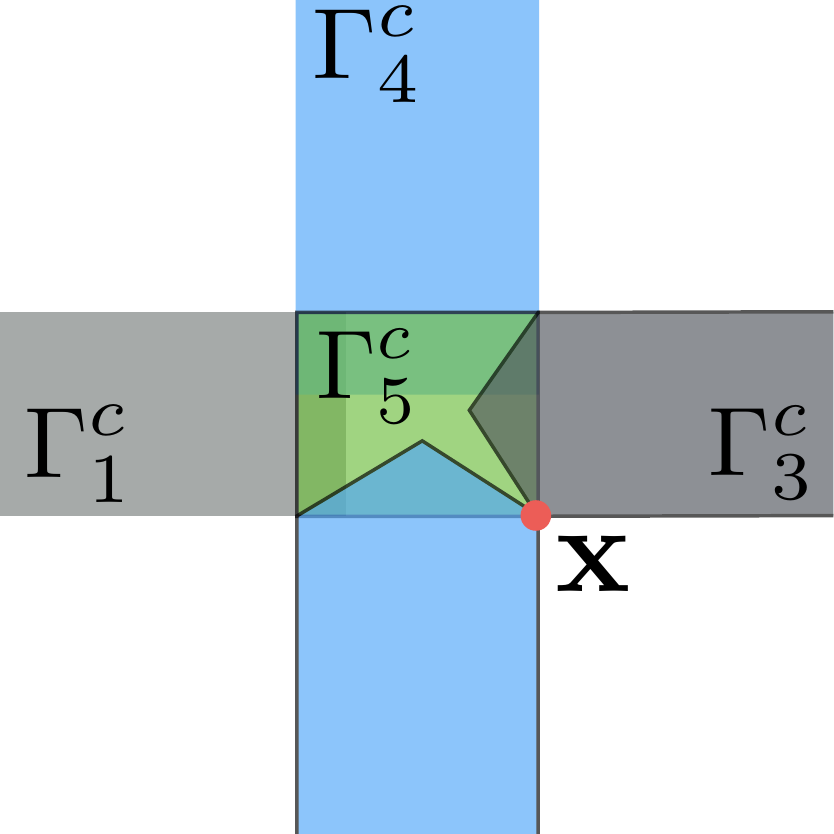}}}
\caption{Representing $\Gamma$ using $\Gamma_i$ and $\phi_i$.}
\label{fig: corner}
\end{center}
\end{figure}

\subsection{Physical Interpretations\label{sec: relationship}}

Many motion planning problems can be formulated into \cref{eq: the problem} when $\mathbf{x}$ is regarded as the trajectory as will be discussed in \cref{sec: mobile robot}. The dimension of the  problem $n$ is proportional to the number of sampling points on the trajectory. If continuous trajectories are considered, then $n\rightarrow \infty$ and $\mathbb{R}^n$ approaches the space of continuous functions $\mathcal{C}(\mathbb{R})$ in the limit.

In addition to motion planning problems, the proposed method deals with any problem with similar geometric properties as specified in \cref{assump: cost} and \cref{assump: constraint}. 
Moreover, problems with global linear equality constraints also fit into the framework if we solve the problem in the low-dimensional linear manifold defined by the linear equality constraints. 
The case for nonlinear equality constraints is much trickier since convexification on nonlinear manifold is difficult in general. A relaxation method to deal with nonlinear equality constraints is discussed in \cite{liu2017geometric}.

\section{Solving the Optimization Problem\label{sec: solve optimization}}

\subsection{The Convex Feasible Set Algorithm}
To solve the problem \cref{eq: the problem} efficiently, we propose the convex feasible set algorithm. As introduced in \cref{sec: benchmark}, a convex feasible set $\mathcal{F}$ for the set $\Gamma$ is  a convex set such that $\mathcal{F}\subset\Gamma$. $\mathcal{F}$ is not unique. We define the desired $\mathcal{F}$ in \cref{sec: find cfs}. As $\Gamma$ can be covered by several (may be infinitely many) convex feasible sets, we can efficiently search the non-convex space $\Gamma$ for solutions by solving a sequence of convex optimizations constrained in a sequence of convex feasible sets. The idea is implemented iteratively as shown in \cref{fig: cfs}. At iteration $k$, given a reference point $\mathbf{x}^{(k)}$, a convex feasible set $\mathcal{F}^{(k)}:=\mathcal{F}(\mathbf{x}^{(k)})\subset \Gamma$ is computed around $\mathbf{x}^{(k)}$. Then a new reference point $\mathbf{x}^{(k+1)}$ will be obtained by solving the resulting convex optimization problem
\begin{equation}
\mathbf{x}^{(k+1)}=\arg\min_{\mathbf{x}\in\mathcal{F}^{(k)}} J(\mathbf{x}).\label{eq: cfs qp}
\end{equation} 

The optimal solution will be used as the reference point for the next step. The iteration will terminate if either the change in solution is small, e.g.,  
\begin{equation}
\|\mathbf{x}^{(k+1)}-\mathbf{x}^{(k)}\|\leq \epsilon_1,\label{eq: terminal condition converge}
\end{equation}
for some small $\epsilon_1>0$, or the descent in cost is small, e.g., 
\begin{equation}
J(\mathbf{x}^{(k)})-J(\mathbf{x}^{(k+1)})\leq\epsilon_2, \label{eq: terminal condition cost}
\end{equation}
for some small $\epsilon_2>0$. We will show in \cref{sec: proofs} that these two conditions are equivalent and both of them imply convergence. 
The process is summarized in \cref{alg: cfs}. 

\begin{algorithm}
\caption{The Convex Feasible Set Algorithm}
\label{alg: cfs}
\begin{algorithmic}
\STATE{Initialize initial guess $\mathbf{x}^{(0)}$, $k:=0$;}
\WHILE{True}
\STATE{Find a convex feasible set $\mathcal{F}^{(k)}\subset \Gamma$ for $\mathbf{x}^{(k)}$;}
\STATE{Solve the convex optimization problem \cref{eq: cfs qp} for $\mathbf{x}^{(k+1)}$;}
\IF{\cref{eq: terminal condition converge} or \cref{eq: terminal condition cost} is satisfied}
\STATE{Break the while loop;}
\ENDIF
\STATE{$k:=k+1$;}
\ENDWHILE
\RETURN $\mathbf{x}^{(k+1)}$;
\end{algorithmic}
\end{algorithm}

\subsection{Finding the Convex Feasible Set\label{sec: find cfs}}

Considering the semi-convex decomposition \cref{eq: gamma}, we try to find a convex feasible set $\mathcal{F}_i$ for each constraint $\Gamma_i=\{\mathbf{x}:\phi_i(\mathbf{x})\geq 0\}$.

\subsubsection*{Case 1: $\phi_i$ is concave} Then $\Gamma_i$ is convex. The convex feasible set is chosen to be itself,
\begin{equation}
\mathcal{F}_i=\Gamma_i.\label{eq: def cfs self}
\end{equation}

\subsubsection*{Case 2: $\phi_i$ is convex} Then $\Gamma_i^c$ is convex. The convex feasible set $\mathcal{F}_i$ with respect to a reference point $\mathbf{x}^r\in\mathbb{R}^n$ is defined as
\begin{equation}
\mathcal{F}_i(\mathbf{x}^r):=\{\mathbf{x}:\phi_i(\mathbf{x}^r)+\hat{\nabla}\phi_i(\mathbf{x}^r)(\mathbf{x}-\mathbf{x}^r)\geq 0\},\label{eq: def cfs convex}
\end{equation}
where $\hat{\nabla}\phi_i(\mathbf{x}^r)\in D\phi_i(\mathbf{x}^r)$ is a sub-gradient. When $\phi_i$ is smooth at $\mathbf{x}^r$, $\hat{\nabla}\phi_i(\mathbf{x}^r)$ equals to the gradient $\nabla\phi_i(\mathbf{x}^r)$. Otherwise, the sub-gradient is chosen according to the method discussed in \cref{sec: subgradient}. 
Since $\phi_i$ is convex, $\phi_i(\mathbf{x})\geq \phi_i(\mathbf{x}^r)+\partial_{\mathbf{x}-\mathbf{x}^r}\phi_i(\mathbf{x}^r) \geq \phi_i(\mathbf{x}^r) + d\cdot(\mathbf{x}-\mathbf{x}^r)$ for all $d\in D\phi_i(\mathbf{x}^r)$ where the second inequality is due to \cref{eq: sub-differential}. Hence $\mathcal{F}_i(\mathbf{x}^r)\subset \{\mathbf{x}:\phi_i(\mathbf{x})\geq 0\}=\Gamma_i$ for all $\mathbf{x}^r\in\mathbb{R}^n$.

\subsubsection*{Case 3: $\phi_i$ is neither concave nor convex}  Considering \cref{eq: semi-convex defn}, the convex feasible set with respect to the reference point $\mathbf{x}^r$ is defined as
\begin{equation}
\mathcal{F}_i(\mathbf{x}^r):=\{\mathbf{x}:\phi_i(\mathbf{x}^r)+\hat{\nabla}\phi_i(\mathbf{x}^r)(\mathbf{x}-\mathbf{x}^r)\geq \frac{1}{2}(\mathbf{x}-\mathbf{x}^r)^T H^*_i(\mathbf{x}-\mathbf{x}^r)\},\label{eq: def cfs ncnc}
\end{equation}
where $\hat{\nabla}\phi_i(\mathbf{x}^r)\in D\phi_i(\mathbf{x}^r)$ is chosen according to the method discussed in \cref{sec: subgradient}. Since $\phi_i$ is semi-convex, $\phi_i(\mathbf{x}) \geq \phi_i(\mathbf{x}^r) + \partial_{\mathbf{x}-\mathbf{x}^r}\phi_i(\mathbf{x}^r) - \frac{1}{2}(\mathbf{x}-\mathbf{x}^r)^TH^{*}_i(\mathbf{x}-\mathbf{x}^r) \geq \phi_i(\mathbf{x}^r) + d\cdot (\mathbf{x}-\mathbf{x}^r) - \frac{1}{2}(\mathbf{x}-\mathbf{x}^r)^TH^{*}_i(\mathbf{x}-\mathbf{x}^r)$ for all $d\in D\phi_i(\mathbf{x}^r)$. Hence $\mathcal{F}_i(\mathbf{x}^r)\subset \{\mathbf{x}:\phi_i(\mathbf{x})\geq 0\}=\Gamma_i$ for all $\mathbf{x}^r\in\mathbb{R}^n$.

Considering \cref{eq: def cfs self}, \cref{eq: def cfs convex} and \cref{eq: def cfs ncnc}, the convex feasible set for $\Gamma$ at $\mathbf{x}^r$ is defined as
\begin{eqnarray}
\mathcal{F}(\mathbf{x}^r):=\bigcap_{i=1}^N\mathcal{F}_i(\mathbf{x}^r).
\end{eqnarray}

\subsection{Choosing the Optimal Sub-Gradients\label{sec: subgradient}}

The sub-gradients in \cref{eq: def cfs convex} and \cref{eq: def cfs ncnc} should be chosen such that the steepest descent of $J$ in the set $\Gamma$ is always included in the convex feasible set $\mathcal{F}$. 

Let $B(\mathbf{x},r)$ denote the unit ball centered at $\mathbf{x}\in\mathbb{R}^n$ with radius $r$. At point $\mathbf{x}^r$, a search direction $v\in\partial B(\mathbf{0},1)$ is feasible if for all $i$, one of the three conditions hold: 
\begin{itemize}
\item $\phi_i(\mathbf{x}^r)> 0$;
\item $\phi_i(\mathbf{x}^r)=0$ and there exists $d\in D\phi_{i}$ such that $v\cdot d\geq 0$;
\item $\phi_i(\mathbf{x}^r)<0$ and there exists $d\in D\phi_{i}$ such that $v\cdot d> 0$. 
\end{itemize}

Define the set of feasible search directions as $C(\mathbf{x}^r)$, which is non empty since we can choose $v$ to be $d/\|d\|$ for any nonzero $d\in D\phi_i$. $D\phi_i$ always contain a nonzero element by the first statement in \cref{assump: regularity}. Then the direction of the steepest descent is $v^*:=\arg\min_{v\in C(\mathbf{x}^r)}\nabla J \cdot v$. If $v^*$ is not unique, the tie breaking mechanism is designed to be: choosing the one with the smallest first entry, the smallest second entry, and so on\footnote{Note that the tie braking mechanism can be any as long as it makes $v^*$ unique. The uniqueness is exploited in \cref{lemma: fixed pt}.}. Then the optimal sub-gradient is chosen to be 
\begin{equation}
\hat{\nabla}\phi_{i}:=\arg\min_{d\in DF_i}\nabla J\cdot d/\|d\|,
\end{equation}
where $DF_i$ is the feasible set of sub-gradients for $\phi_i$ such that $DF_i := D\phi_i$ when $\phi_i > 0$; $DF_i:=\{d\in D\phi_{i}|d\cdot v^*\geq 0\}$ when $\phi_i = 0$; and $DF_i := \{d\in D\phi_{i}|d\cdot v^*> 0\}$ when $\phi_i < 0$. The set $DF_i$ is non empty by definition of $C(\mathbf{x}^r)$. To avoid singularity, let $d/\|d\|$ be $\mathbf{0}$ when $d = \mathbf{0}$. 
\cref{fig: sub-gradient} illustrates the above procedure in choosing the optimal sub-gradient, where the short arrow shows the direction of the steepest descent of $J$, the shaded sector shows the range of sub-differentials, the long arrow denotes the optimal sub-gradient, and the shaded half-space is the convex feasible set $\mathcal{F}_i$. In case one, $v^*$ is in the same direction of $\hat{\nabla}\phi_i$, while the two are perpendicular to each other in case two.

\begin{figure}[t]
\begin{center}
\subfloat[Case one.]{\includegraphics[width=6cm]{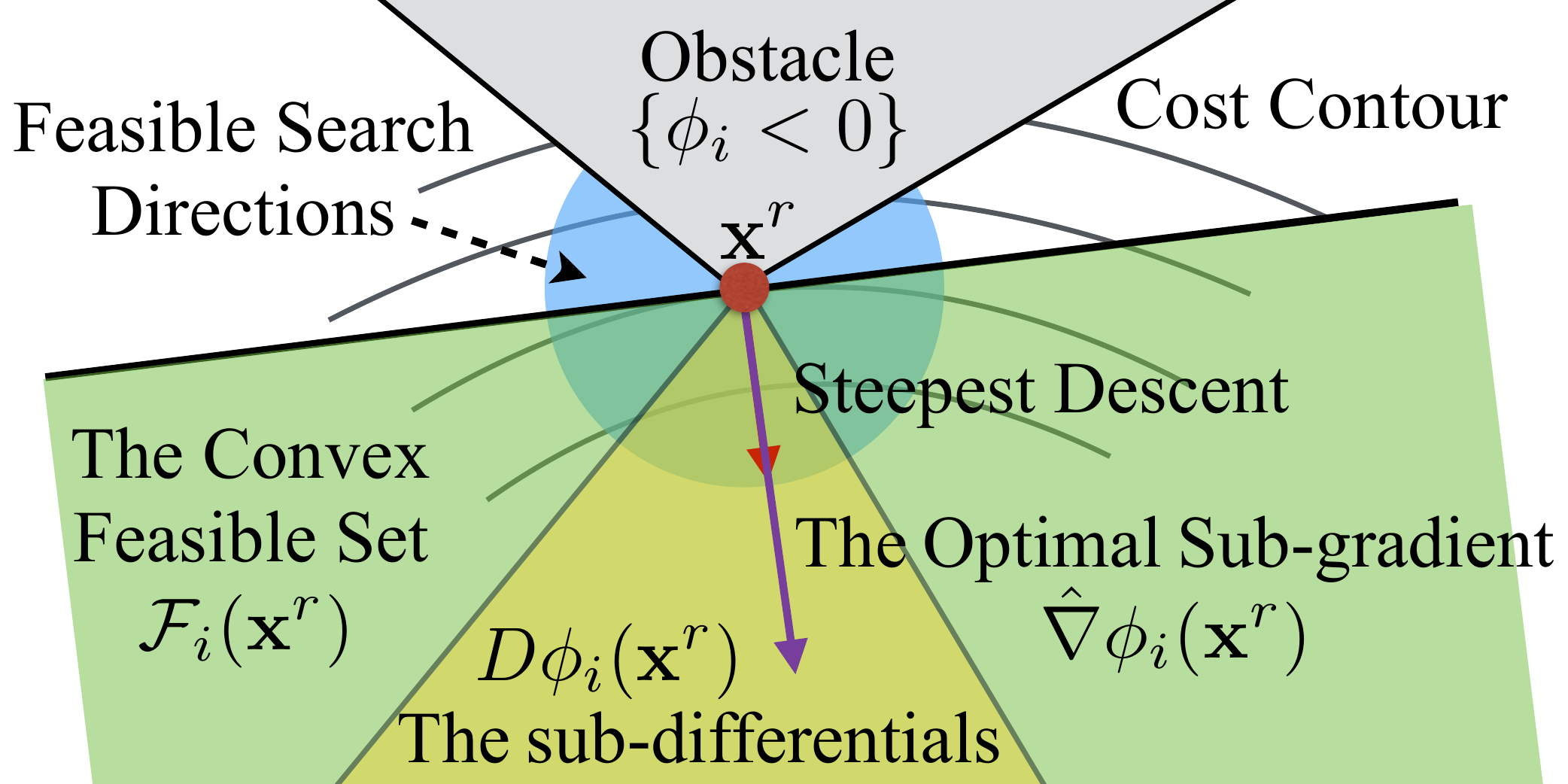}}~~
\subfloat[Case two.]{\includegraphics[width=5.5cm]{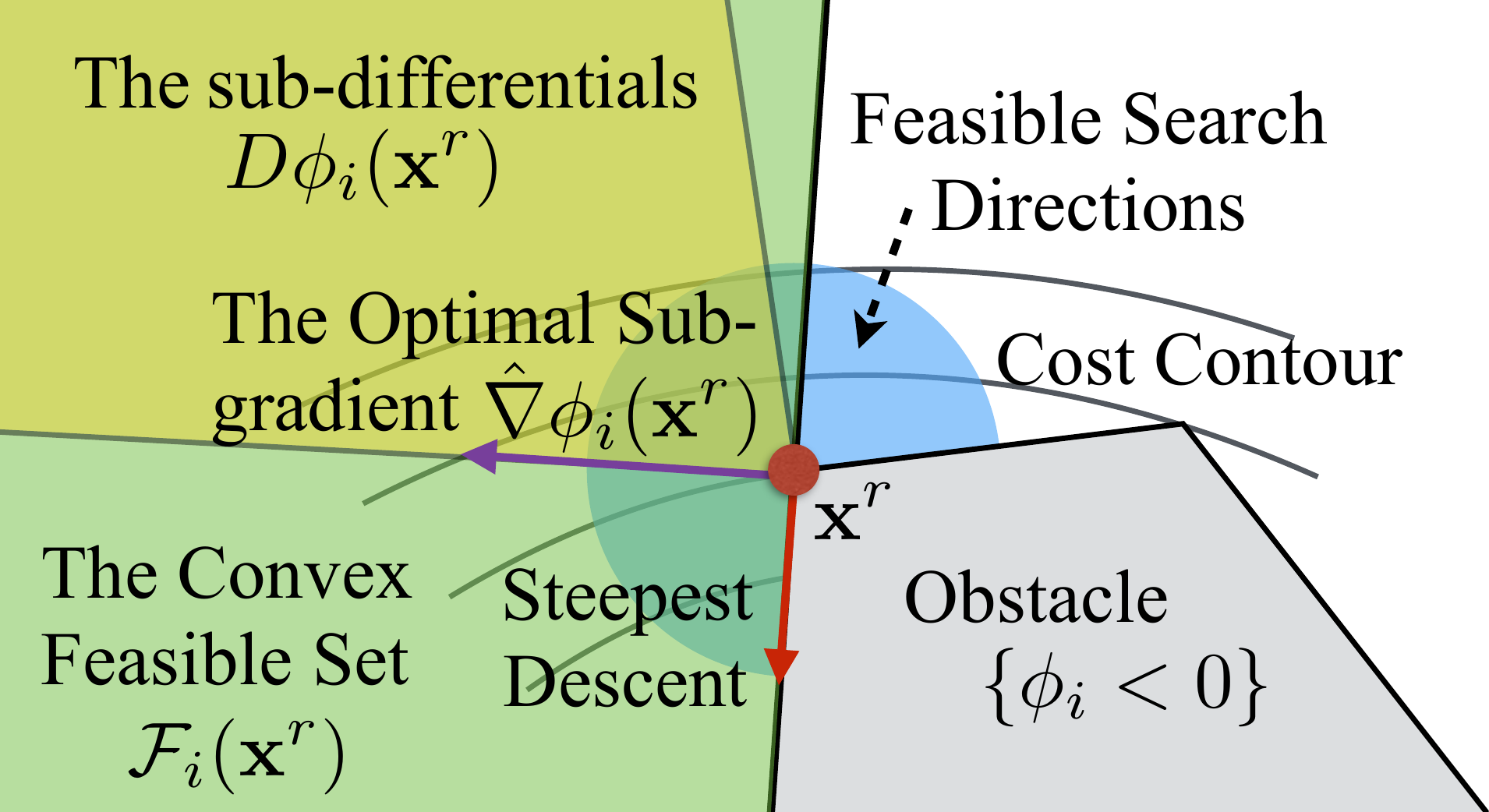}}
\caption{The choice of sub-gradient $\hat{\nabla}\phi_i(\mathbf{x}^r)$ on non-smooth point $\mathbf{x}^r$.}
\label{fig: sub-gradient}
\end{center}
\end{figure}

\section{Properties of the Convex Feasible Set Algorithm\label{sec: proofs}}

This section shows the feasibility and convergence of \cref{alg: cfs}. The main result is summarized in the following theorem:
\begin{theorem}[Convergence of \cref{alg: cfs}]
Under \cref{alg: cfs}, the sequence $\{\mathbf{x}^{(k)}\}$ will converge to some $\mathbf{x}^*\in\Gamma$ for any initial guess $\mathbf{x}^{(0)}$ such that $\mathcal{F}^{(0)}\neq\emptyset$. $\mathbf{x}^*$ is a strong local optimum of \cref{eq: the problem} if the limit is attained, i.e., there exists a constant $K\in\mathbb{N}$ such that $\mathbf{x}^{(k)}=\mathbf{x}^*$ for all $k>K$. $\mathbf{x}^*$ is at least a weak local optimum of \cref{eq: the problem} if the limit is not attained. \label{thm: converge}
\end{theorem}

We say that $\mathbf{x}^*$ is a strong local optimum of \cref{eq: the problem} if $J$ is nondecreasing along any feasible search direction, e.g.,  $\nabla J(\mathbf{x}^*)v\geq0$ for all $v\in C(\mathbf{x}^*)$ as shown in \cref{fig: strong local optimum}. We say that $\mathbf{x}^*$ is a weak local optimum of \cref{eq: the problem} if the KKT condition is satisfied for some sub-gradients, i.e., $\nabla J(\mathbf{x}^*)+\sum_{i=1}^N\lambda_id_i=0$ for some $d_i\in D\phi_i(\mathbf{x}^*)$ as shown in \cref{fig: weak local optimum}. $\lambda_i$ is a Lagrange multiplier such that $\lambda_i\leq 0$ and $\lambda_i\phi_i(\mathbf{x})=0$ (complementary slackness) for all $i=1,\cdots,N$. A strong local optimum is always a weak local optimum. The two are equivalent when all $\phi_i$'s are smooth at $\mathbf{x}^*$. 

\begin{figure}[t]
\begin{center}
\subfloat[A strong local optimum.\label{fig: strong local optimum}]{\includegraphics[width=3.8cm]{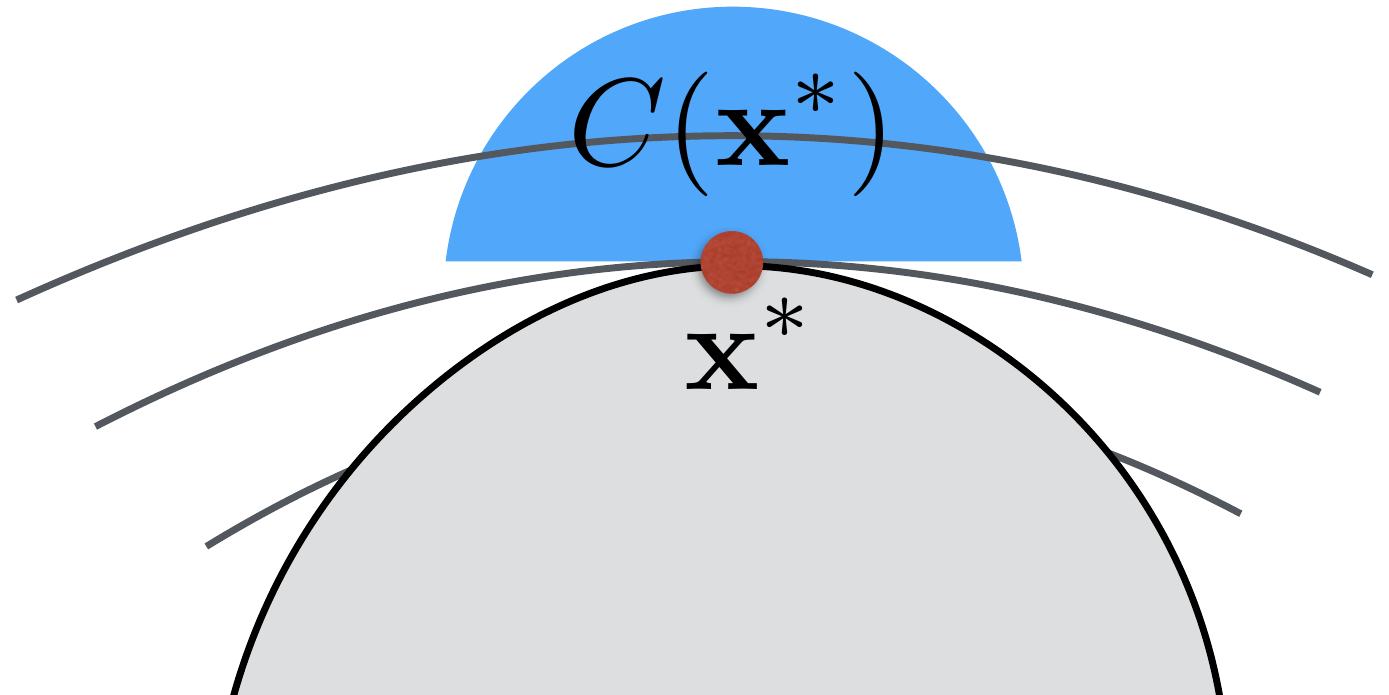}}~~~
\subfloat[A weak local optimum.\label{fig: weak local optimum}]{\includegraphics[width=3.8cm]{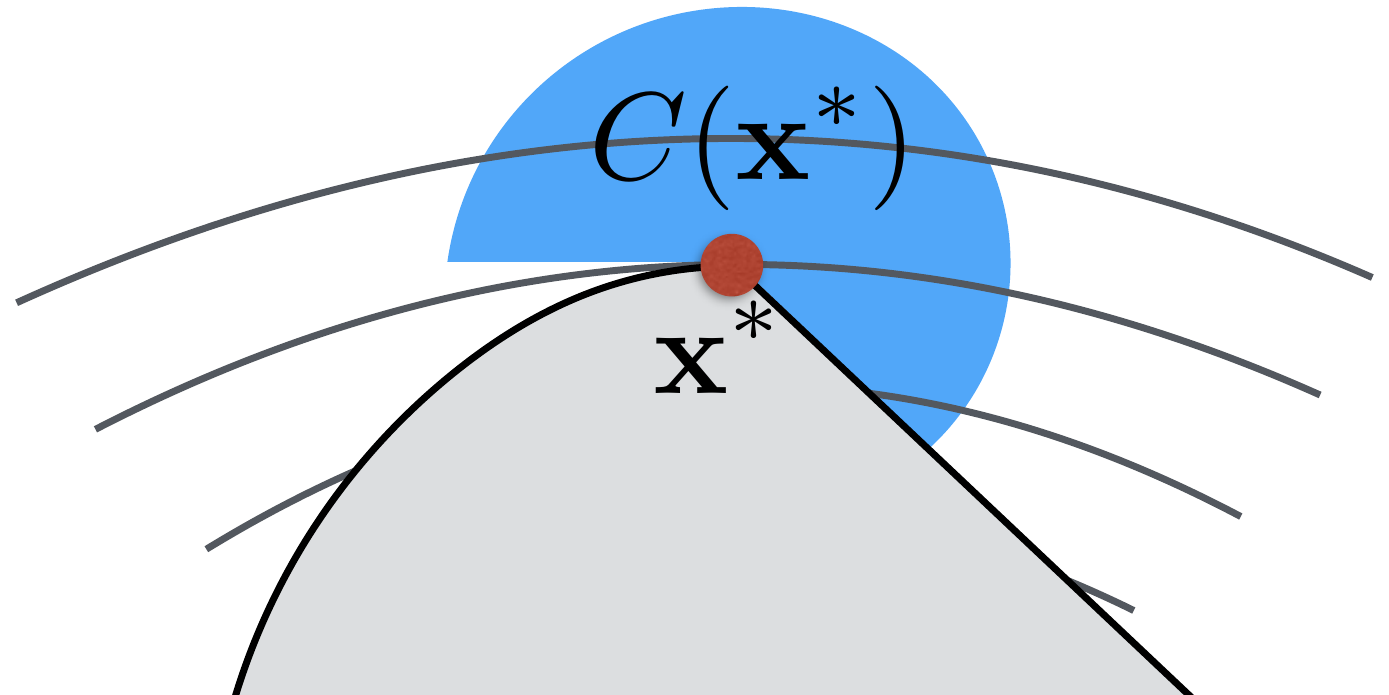}}
\caption{Definition of local optima.}
\label{fig: local optima}
\end{center}
\end{figure}

\subsection{Preliminary Results}
Before proving \cref{thm: converge}, we present some preliminary results that are useful toward proving the theorem. \cref{lemma: existence} states that given a feasible reference point $\mathbf{x}^r$, $\mathcal{F}(\mathbf{x}^r)$ is a convex set containing $\mathbf{x}^r$ with nontrivial interior. This conclusion naturally leads to the hypothesis that a suboptimal reference $\mathbf{x}^r$ can be improved by optimizing $J$ in the convex feasible set $\mathcal{F}(\mathbf{x}^r)$. We will show in \cref{lemma: fixed pt} that if the solution can not be improved using \cref{alg: cfs}, then $\mathbf{x}^r$ is already a strong local optimum of \cref{eq: the problem}. Otherwise, if we can keep improving the result using \cref{alg: cfs}, this process will generate a Cauchy sequence that converges to a weak local optimum of \cref{eq: the problem}, which will be shown in \cref{lemma: cauchy}. Given these results, the conclusion in the theorem follows naturally.

The interior of a set $S$ is denoted as $S^o$. We say that a reference point $\mathbf{x}^r\in\mathbb{R}^n$ is feasible if $\mathbf{x}^r\in \Gamma$; and 
$\mathbf{x}^*\in\Gamma$ is a fixed point of \cref{alg: cfs} if
\begin{equation}
\mathbf{x}^* = \arg\min_{\mathbf{x}\in\mathcal{F}{(\mathbf{x}^*)}}J(\mathbf{x}).
\end{equation}

\begin{lemma} [Feasibility]
If $\mathbf{x}^r\in\Gamma$, then $\mathbf{x}^r\in \mathcal{F}(\mathbf{x}^r)$ and $\mathcal{F}^o(\mathbf{x}^r)\neq \emptyset$. \label{lemma: existence}
\end{lemma}
\begin{proof}
When $\mathbf{x}^r$ is feasible, $\mathbf{x}^r\in \mathcal{F}_i(\mathbf{x}^r)$ for all $i$ according to the definitions in \cref{eq: def cfs self}, \cref{eq: def cfs convex} and \cref{eq: def cfs ncnc}. Hence $\mathbf{x}^r\in \mathcal{F}(\mathbf{x}^r)$.

Claim 1: if $\mathbf{x}^r\in\Gamma^o$, then $\mathbf{x}^r\in\mathcal{F}^o(\mathbf{x}^r)$.  
The condition $\mathbf{x}^r\in\Gamma^o$ implies that $\phi_i(\mathbf{x}^r)> 0$ for all $i=1,\cdots, N$. Then the inequality in \cref{eq: def cfs convex} and \cref{eq: def cfs ncnc} are not tight at $\mathbf{x}^r$. Hence $\mathbf{x}^r\in\mathcal{F}_i^o(\mathbf{x}^r)$ in either of the three cases. Since $N$  is finite, $\mathbf{x}^r\in\bigcap_{i=1}^N \mathcal{F}_i^o(\mathbf{x}^r) = \mathcal{F}^o(\mathbf{x}^r)$.

Claim 2: if $\mathbf{x}^r\in\partial\Gamma$, there exists a non trivial $v\in\mathbb{R}^n$ such that $\mathbf{x}^r+v\in\mathcal{F}^o(\mathbf{x}^r)$. 
Let  $I:=\{i:\phi_i(x^r)=0\}$. By the third statement in \cref{assump: regularity}, 
there exists a unit vector $w\in\mathbb{R}^n$ such that $\partial_w\phi_i(\mathbf{x}^r) < 0$ for all $i\in I$.  
Fix $i\in I$. Let $a>0$ be sufficiently small. When $\phi_i$ is concave, 
\begin{equation*}
\phi_i(\mathbf{x}^r-a w)\geq \partial_{-aw}\phi_i(\mathbf{x}^r) -\frac{a^2}{2} w^T H_i^* w\geq -a \partial_w\phi_i(\mathbf{x}^r) - \frac{a^2}{2} w^T H_i^* w > 0,
\end{equation*}
where the first inequality is due to semi-convexity, the second inequality is due to \cref{eq: directional derivative}, and the third inequality is because $a$ is small. When $\phi_i$ is convex, 
\begin{equation*}
\phi_i(\mathbf{x}^r) + \hat{\nabla}\phi_i(\mathbf{x}^r)\cdot(-aw) \geq -a \partial_{w}\phi_i(\mathbf{x}^r) > 0,
\end{equation*}
where the first inequality is due to $\hat{\nabla}\phi_i(\mathbf{x}^r)\cdot w\leq \partial_w \phi_i(\mathbf{x}^r)$ by definition \cref{eq: sub-differential}. When $\phi_i$ is neither concave nor convex, 
\begin{equation*}
\phi_i(\mathbf{x}^r) + \hat{\nabla}\phi_i(\mathbf{x}^r)\cdot(-aw) \geq -a \partial_{w}\phi_i(\mathbf{x}^r) > \frac{a^2}{2} w^T H_i^* w.
\end{equation*}
Hence $\mathbf{x}^r - aw \in \mathcal{F}^o_i(\mathbf{x}^r)$ for any sufficiently small $a$. Since $I$ is finite, we can find a constant $\epsilon$ such that $\mathbf{x}^r - aw \in \mathcal{F}^o_i(\mathbf{x}^r)$ for all $i\in I$ and $0<a\leq\epsilon$. 
On the other hand, $\mathbf{x}^r\in \Gamma_j^o$ for all $j\notin I$. According to Claim 1, $\mathbf{x}^r\in \bigcap_{j\notin I}\mathcal{F}_j^o(\mathbf{x}^r)$. There exists a constant $\epsilon_0>0$ such that $B(\mathbf{x}^r,\epsilon_0)\subset\bigcap_{j\notin I}\mathcal{F}_j^o(\mathbf{x}^r)$. Define $v = -\min(\epsilon_0,\epsilon) w$. According to previous discussion, $\mathbf{x}^r+v\in \mathcal{F}^o(\mathbf{x}^r)$.

Claim 1 and Claim 2 imply that $\mathcal{F}(\mathbf{x}^r)$ has nonempty interior.
\end{proof}

\begin{proposition}[Fixed point]
If $\mathbf{x}^*$ is a fixed point of \cref{alg: cfs}, then $\mathbf{x}^*$ is a strong local optimum of \cref{eq: the problem}.\label{lemma: fixed pt}
\end{proposition}
\begin{proof}
We need to show that $\nabla J(\mathbf{x}^*)\cdot v\geq 0$ for all $v\in C(\mathbf{x}^*)$. If $\nabla J(\mathbf{x}^*)=0$, then $\mathbf{x}^*$ is the global optimum of \cref{eq: the problem}. Consider the case $\nabla J(\mathbf{x}^*)\neq0$.  
Claim that $\mathbf{x}^*\in\partial\Gamma$.
First of all, since $\mathbf{x}^*\in \mathcal{F}(\mathbf{x}^*) \subset \Gamma$, $\mathbf{x}^*$ is a feasible point.
Moreover, since $J$ is strictly convex, the optimal point $\mathbf{x}^*$ must be on the boundary of $\mathcal{F}(\mathbf{x}^*)$, i.e., $\mathbf{x}^*\in\partial\mathcal{F}_i(\mathbf{x}^*)$ for some $i$. According to \cref{eq: def cfs self}, \cref{eq: def cfs convex} and \cref{eq: def cfs ncnc}, $\partial\mathcal{F}_i(\mathbf{x}^*)$ equals to $\{\mathbf{x}:\phi_i(\mathbf{x})=0\}$ in case~1, $\{\mathbf{x}:\phi_i(\mathbf{x}^*) + \hat{\nabla}\phi_i(\mathbf{x}^*)(\mathbf{x}-\mathbf{x}^*) = 0\}$ in case 2, and $\{\mathbf{x}:\phi_i(\mathbf{x}^*) + \hat{\nabla}\phi_i(\mathbf{x}^*)(\mathbf{x}-\mathbf{x}^*) = \frac{1}{2}(\mathbf{x}-\mathbf{x}^*)^T H^*_i(\mathbf{x}-\mathbf{x}^*)\}$ in case 3. Then $\mathbf{x}^*\in\partial\mathcal{F}_i(\mathbf{x}^*)$ implies that $\phi_i(\mathbf{x}^*)=0$. Hence $\mathbf{x}^*\in\partial\Gamma$. Let $I = \{i:\phi_i(\mathbf{x}^*)=0\} = \{i:\mathbf{x}^*\in\partial \mathcal{F}_i(\mathbf{x}^*)\}$. 

Consider $v^*:=\arg\min_{v\in C(\mathbf{x}^*)} \nabla J(\mathbf{x}^*)\cdot v$. If the minimum is not unique, use the tie breaking mechanism discussed in \cref{sec: subgradient}. 
Claim that $\hat{\nabla} \phi_i\cdot v^*\geq 0$ for all $i\in I$. For $i\in I$ such that $\phi_i$ is smooth at $\mathbf{x}^*$, the definition of $C(\mathbf{x}^*)$ implies that $\hat{\nabla} \phi_i(\mathbf{x}^*)\cdot v^* = \nabla \phi_i(\mathbf{x}^*)\cdot v^*\geq 0$. For $i\in I$ such that $\phi_i$ is not smooth at $\mathbf{x}^*$, $\hat{\nabla} \phi_i (\mathbf{x}^*)\cdot v^*\geq 0$ since $\hat{\nabla}\phi_i\in DF_i:=\{d\in D\phi_{i}|d\cdot v^*\geq 0\}$. On the other hand, since $\mathbf{x}^*$ is the optimal solution of the smooth optimization $\min_{\mathbf{x}\in\mathcal{F}(\mathbf{x}^*)}J(\mathbf{x})$ , the KKT condition is satisfied,
\begin{equation*}
\nabla J(\mathbf{x}^*)+\sum_{i=1}^{N}\lambda_i\hat{\nabla} \phi_i(\mathbf{x}^*) = 0.
\end{equation*}
The complementary slackness condition implies that $\lambda_i\leq 0$ for $i\in I$ and $\lambda_j=0$ for $j\notin I$. Hence
\begin{equation*}
\nabla J(\mathbf{x}^*)\cdot v^* = - \sum_{i=1}^{N}\lambda_i\hat{\nabla} \phi_i(\mathbf{x}^*)\cdot v^* = -\sum_{i\in I}\lambda_i\hat{\nabla} \phi_i(\mathbf{x}^*)\cdot v^* \geq 0.
\end{equation*}
Thus $J(\mathbf{x}^*)\cdot v\geq 0$ for all $v\in C(\mathbf{x}^*)$. So $\mathbf{x}^*$ is a strong local optimum of \cref{eq: the problem}.
\end{proof}

\begin{remark}
\cref{lemma: existence} and \cref{lemma: fixed pt} imply that a feasible $\mathbf{x}^r$ can always be improved by optimizing over the convex feasible set $\mathcal{F}(\mathbf{x}^r)$ if $\mathbf{x}^r$ itself is not a strong local optimum. 
However, the existence of nonempty convex feasible set for an infeasible reference point is more intricate, which is deeply related to the choice of the functions $\phi_i$'s. The design considerations of $\phi_i$'s such that a nonempty convex feasible set always exists will be addressed in \cref{sec: mobile robot}. 
\label{remark: existence}
\end{remark}

\begin{lemma}[Strong descent]
For any feasible $\mathbf{x}^{(k)}$, the descent of the objective function satisfies that $\nabla J(\mathbf{x}^{(k+1)})(\mathbf{x}^{(k)}-\mathbf{x}^{(k+1)})\geq 0$. Moreover, if $J(\mathbf{x}^{(k+1)})=J(\mathbf{x}^{(k)})$, then $\mathbf{x}^{(k+1)}=\mathbf{x}^{(k)}$.\label{lem: strong descent}
\end{lemma}

\begin{figure}[t]
\begin{center}
\includegraphics[width=3.3cm]{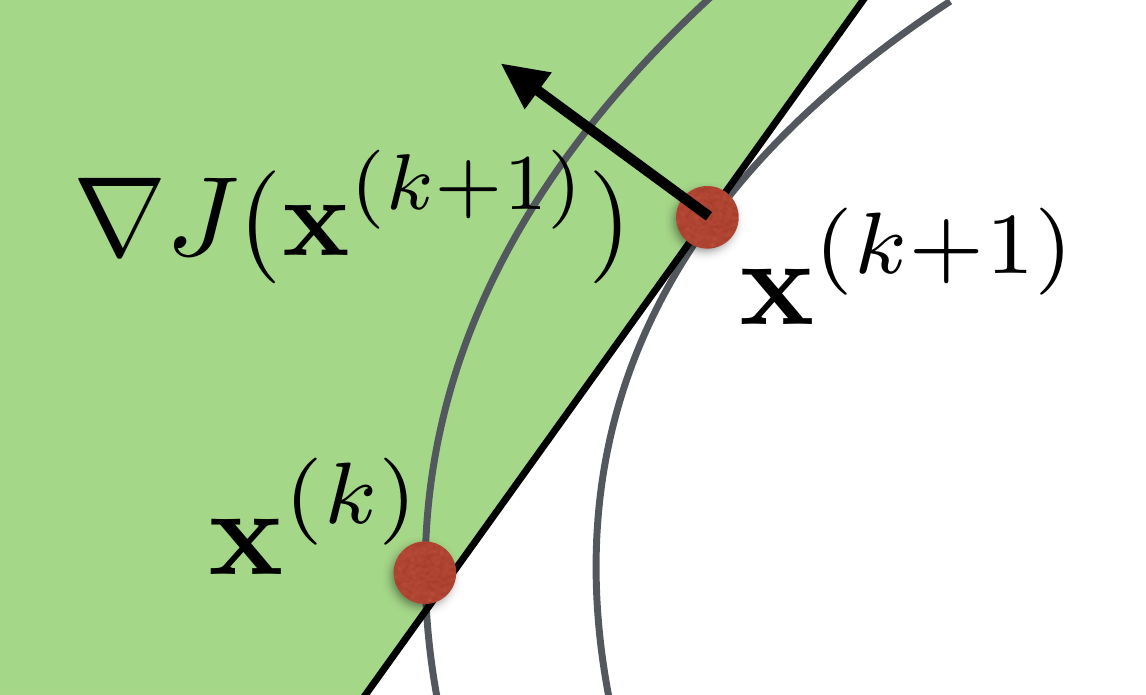}
\caption{The descent of the sequence $\{\mathbf{x}^{(k)}\}$.}
\label{fig: descent}
\end{center}
\end{figure}

\begin{proof}
Claim that $\mathcal{F}(\mathbf{x}^{(k)})$ is a subset of the half space $H:=\{\mathbf{x}\mid\nabla J(\mathbf{x}^{(k+1)})\cdot (\mathbf{x}-\mathbf{x}^{(k+1)})\geq 0\}$ as shown by the shaded area in \cref{fig: descent}. If not, there must be some $\hat{\mathbf{x}}\in \mathcal{F}(\mathbf{x}^{(k)})$ such that $\nabla J(\mathbf{x}^{(k+1)})\cdot(\hat{\mathbf{x}}-\mathbf{x}^{(k+1)})< 0$. Let $v:= \hat{\mathbf{x}}-\mathbf{x}^{(k+1)}$. Since $\mathcal{F}(\mathbf{x}^{(k)})$ is convex, then $\mathbf{x}^{(k+1)} + av\in \mathcal{F}(\mathbf{x}^{(k)})$ for $a\in[0,1]$. Since $J$ is smooth,  there exists a positive constant $c>0$ such that $J(\mathbf{x}^{(k+1)} + av) \leq J(\mathbf{x}^{(k+1)}) + a\nabla J(\mathbf{x}^{(k+1)})\cdot v + c a^2 \|v\|^2$. When $a$ is sufficiently small, the right-hand side of the inequality is strictly smaller than $J(\mathbf{x}^{(k+1)})$. Then $J(\mathbf{x}^{(k+1)} + av) < J(\mathbf{x}^{(k+1)})$, which contradicts with the fact that $\mathbf{x}^{(k+1)}$ is the minimum of $J$ in the convex feasible set $\mathcal{F}(\mathbf{x}^{(k)})$. 
Hence the claim is true. Since $\mathbf{x}^{(k)}\in \mathcal{F}(\mathbf{x}^{(k)})\subset H$, then $\nabla J(\mathbf{x}^{(k+1)})(\mathbf{x}^{(k)}-\mathbf{x}^{(k+1)})\geq 0$. 
Moreover, since $J$ is strictly convex, $J(\mathbf{x})> J(\mathbf{x}^{(k+1)})$ for all $\mathbf{x}\in H\setminus\{\mathbf{x}^{(k+1)}\}$. Hence $J(\mathbf{x}^{(k+1)})=J(\mathbf{x}^{(k)})$ implies $\mathbf{x}^{(k)}=\mathbf{x}^{(k+1)}$.
\end{proof}

\begin{proposition}[Convergence of strictly descending sequence]
Consider the sequence $\{\mathbf{x}^{(k)}\}$ generated by \cref{alg: cfs}. If $J(\mathbf{x}^{(1)})> J(\mathbf{x}^{(2)}) >\cdots$, then the sequence $\{\mathbf{x}^{(k)}\}$ converges to a weak local optimum $\mathbf{x}^*$ of \cref{eq: the problem}.\label{lemma: cauchy}
\end{proposition}
\begin{proof}
The monotone sequence $\{J(\mathbf{x}^{(k)})\}_{i=2}^{\infty}$ converges to some value $a\geq \min J$. If $a=\min J$, the sequence $\{\mathbf{x}^{(k)}\}$ converges to the global optima. Consider the case $a>\min J$. Since $J$ is strictly convex by \cref{assump: cost}, the set $\{\mathbf{x}:J(\mathbf{x})\leq J(\mathbf{x}^{(1)})\}$ is compact. Then there exists a subsequence of $\{\mathbf{x}^{(k)}\}$ that converges to $\mathbf{x}^*$ such that $J(\mathbf{x}^*)=a$. Since $\Gamma$ is closed, $\mathbf{x}^*\in \Gamma$. 

We need to show that the whole sequence $\{\mathbf{x}^{(k)}\}$ converges to $\mathbf{x}^*$. Suppose not, then there exists $\delta>0$ such that $\forall K>0$, there exists $k>K$ s.t. $\|\mathbf{x}^{(k)}-\mathbf{x}^*\|\geq \delta$. For any $\epsilon\in(0,\delta)$, there exists $k,j\in\mathbb{N}$ such that $\|\mathbf{x}^{(k)}-\mathbf{x}^*\|\leq \epsilon$ and $\|\mathbf{x}^{(k+j)}-\mathbf{x}^*\|\geq \delta$. Since $J$ is strictly convex, there exists $c>0$ such that
\begin{eqnarray*}
J(\mathbf{x}^{(k)}) &\geq& J(\mathbf{x}^{(k+1)})+\nabla J(\mathbf{x}^{(k+1)})(\mathbf{x}^{(k)}-\mathbf{x}^{(k+1)})+c\|\mathbf{x}^{(k)}-\mathbf{x}^{(k+1)}\|^2\\
&\geq& J(\mathbf{x}^{(k+1)}) + c\|\mathbf{x}^{(k)}-\mathbf{x}^{(k+1)}\|^2\\
&\geq& J(\mathbf{x}^{(k+j)}) + c\|\mathbf{x}^{(k+j-1)}-\mathbf{x}^{(k+j)}\|^2 + \cdots + \|\mathbf{x}^{(k)}-\mathbf{x}^{(k+1)}\|^2\\
&\geq& J(\mathbf{x}^{(k+j)}) + c\|\mathbf{x}^{(k)}-\mathbf{x}^{(k+j)}\|^2\\
&\geq& J(\mathbf{x}^{(k+j)}) + c(\|\mathbf{x}^{(k)}-\mathbf{x}^*\|-\|\mathbf{x}^{(k+j)}-\mathbf{x}^*\|)^2\\
&\geq& a + c(\delta-\epsilon)^2,
\end{eqnarray*}
which contradicts with the fact that $J(\mathbf{x}^{(k)})\rightarrow a$ as $\epsilon\rightarrow 0$. Note that the second inequality is due to $\nabla J(\mathbf{x}^{(k+1)})(\mathbf{x}^{(k)}-\mathbf{x}^{(k+1)})\geq 0$ in \cref{lem: strong descent}. The third inequality is by induction. The fourth inequality and the fifth inequality follow from $\Delta$-inequality. Hence we conclude that $\lim_{k\rightarrow\infty}\mathbf{x}^{(k)}=\mathbf{x}^*$. 

Then we need to show that $\mathbf{x}^*$ is a weak local optimum. The proof can be divided into two steps. First, we show that there is a subsequence of the convex feasible sets $\{\mathcal{F}^{(k)}\}$ that converges point-wise to a suboptimal convex feasible set $\mathcal{G}$ at point $\mathbf{x}^*$. Sub-optimality of $\mathcal{G}$ means that the sub-gradients are not chosen according to \cref{sec: subgradient}. Then we show that $\mathbf{x}^*$ is the minimum of $J$ in $\mathcal{G}$. Thus the KKT condition is satisfied at $\mathbf{x}^*$ and $\mathbf{x}^*$ is a weak local optimum of \cref{eq: the problem}.

Consider any $\phi_i$. 
If $\phi_i$ is smooth at $\mathbf{x}^*$, then it is smooth in a neighborhood of $\mathbf{x}^*$ as $\phi_i$ is assumed to be piece-wise smooth. Then $\hat{\nabla}\phi_i(\mathbf{x}^{(k)})$ converges to $d_i :=\nabla\phi_i(\mathbf{x}^{*})$. If $\phi_i$ is not smooth at $\mathbf{x}^*$, it is still locally Lipschitz due to semi-convexity. Then $\hat{\nabla}\phi_i$ is locally bounded. Hence there is a subsequence $\{\hat{\nabla}\phi_i(\mathbf{x}^{(k_j)})\}_{k_j\in\mathbb{N}}$ that converges to some $d_i\in\mathbb{R}^n$.  Claim that $d_i\in D\phi_i(\mathbf{x}^{*})$. By definition \cref{eq: sub-differential}, for any $v\in \mathbb{R}^n$, $\hat{\nabla}\phi_i(\mathbf{x}^{(k_j)})\cdot v\leq \partial_{v}\phi_i(\mathbf{x}^{(k_j)})$. Since $\phi_i$ is piecewise smooth, then either $\partial_{v}\phi_i(\mathbf{x}^{(k_j)})\rightarrow \partial_{v}\phi_i(\mathbf{x}^{*})$ or $\partial_{v}\phi_i(\mathbf{x}^{(k_j)})\rightarrow -\partial_{-v}\phi_i(\mathbf{x}^{*})$. Since $-\partial_{-v}\phi_i(\mathbf{x}^{*})\leq \partial_{v}\phi_i(\mathbf{x}^{*})$ by \cref{eq: directional derivative}, we have the following inequality,
\begin{equation*}
d_i \cdot v = \lim_{k_j\rightarrow\infty }\hat{\nabla}\phi_i(\mathbf{x}^{(k_j)})\cdot v \leq \lim_{k_j\rightarrow\infty }\partial_{v}\phi_i(\mathbf{x}^{(k_j)}) \leq \partial_{v}\phi_i(\mathbf{x}^{*}).
\end{equation*}
Hence by definition \cref{eq: sub-differential}, $d_i\in D\phi_i(\mathbf{x}^*)$.
Then we can choose a subsequence $\{\mathbf{x}^{(k_n)}\}_{k_n\in\mathbb{N}}$ such that $\phi_i(\mathbf{x}^{(k_n)})$ converges to $\phi_i(\mathbf{x}^{*})$ and $\hat{\nabla}\phi_i(\mathbf{x}^{(k_n)})$ converges to $d_i\in D\phi_i(\mathbf{x}^{*})$ for all $i = 1,\cdots, N$. For simplicity and without loss of generality, we use the same notation for the subsequence as the original sequence in the following discussion.

Define a new convex feasible set $\mathcal{G}_i$ such that $\mathcal{G}_i = \Gamma_i$ if $\phi_i$ is concave,  $\mathcal{G}_i = \{\mathbf{x}:\phi_i(\mathbf{x}^*)+d_i(\mathbf{x}-\mathbf{x}^*)\geq 0\}$ if $\phi_i$ is convex, and $\mathcal{G}_i = \{\mathbf{x}:\phi_i(\mathbf{x}^*)+d_i(\mathbf{x}-\mathbf{x}^*)\geq \frac{1}{2}(\mathbf{x}-\mathbf{x}^*)^T H^*_i(\mathbf{x}-\mathbf{x}^*)\}$ if otherwise. 
Let $\mathcal{G} = \bigcap_{i=1}^N \mathcal{G}_i$. 

Claim that $\mathcal{G}=\lim_{k\rightarrow \infty} \mathcal{F}^{(k)}$ where $\lim_{k\rightarrow \infty} \mathcal{F}^{(k)}:= \bigcap_{k\rightarrow\infty}\bigcup_{j = k}^{\infty} \mathcal{F}^{(k)}$. Consider $z\in \mathcal{G}$. If $\phi_i$ is concave, then $z\in \mathcal{F}_i^{(k)} : = \mathcal{F}_i(\mathbf{x}^{(k)})$ for all $k$ according to \cref{eq: def cfs convex}. If $\phi_i$ is convex, then 
\begin{equation*}
\lim_{k\rightarrow\infty}\left[\phi_i(\mathbf{x}^{(k)}) + \hat{\nabla}\phi_i(\mathbf{x}^{(k)})(z - \mathbf{x}^{(k)})\right] = \phi_i(\mathbf{x}^{*}) + d_i(z - \mathbf{x}^{*}) \geq 0,
\end{equation*}
which implies that $z \in \lim_{k\rightarrow\infty} \mathcal{F}_i^{(k)}$. Similarly, we can show that if $\phi_i$ is neither convex or concave, $z$ also lies in the limit of $\mathcal{F}_i^{(k)}$. Hence $z\in\bigcap_{i = 1}^N \lim_{k\rightarrow \infty} \mathcal{F}_i^{(k)} = \lim_{k\rightarrow \infty} \mathcal{F}^{(k)}$. Since $z$ is arbitrary, $\mathcal{G}\subset\lim_{k\rightarrow \infty} \mathcal{F}^{(k)}$. For any $z\in \lim_{k\rightarrow \infty} \mathcal{F}^{(k)}$, then there exists a sequence $\{z_k\}$ that converges to $z$ such that $z_k\in \mathcal{F}^{(k)}$. For any $i$, if $\phi_i$ is concave, then $z_k\in\mathcal{F}_i^{(k)} =\mathcal{G}_i$. Since $\mathcal{G}_i$ is closed, $z\in \mathcal{G}_i$. If $\phi_i$ is convex, then 
\begin{align*}
 \phi_i(\mathbf{x}^{*}) + d_i(z - \mathbf{x}^{*}) = \lim_{k\rightarrow\infty}\left[\phi_i(\mathbf{x}^{(k)}) + \hat{\nabla}\phi_i(\mathbf{x}^{(k)})(z_k - \mathbf{x}^{(k)})\right] \geq 0,
\end{align*}
which implies that $z\in\mathcal{G}_i$. Similarly, we can show that $z\in\mathcal{G}_i$ if $\phi_i$ is neither convex or concave. Hence $z\in\mathcal{G}$. And we verify that $\mathcal{G}=\lim_{k\rightarrow \infty} \mathcal{F}^{(k)}$.

Claim that $\mathbf{x}^* = \arg\min_{\mathbf{x}\in\mathcal{G}}J(\mathbf{x})$. 
Suppose not, then there exists $z\in\mathcal{G}$ such that $J(z)<J(\mathbf{x}^*)$. For all $\epsilon>0$, there exists $y\in\mathcal{F}^{(k)}$ for some $k\in\mathbb{N}$ such that $\|z-y\|<\epsilon$. Since $J$ is smooth, then $J(y) - J(z) < O(\epsilon)$. Thus $J(y) < J(\mathbf{x}^*)$ when $\epsilon$ is sufficiently small. This contradicts with $J(y) > J(\mathbf{x}^{(k)}) > J(\mathbf{x}^*)$. Hence $J(z)\geq J(\mathbf{x}^*)$ for all $z\in\mathcal{G}$. If there exists $z\in\mathcal{G}\setminus\{\mathbf{x}^*\}$ such that $J(z)=J(\mathbf{x}^*)$, then $\frac{z+\mathbf{x}^*}{2}\in \mathcal{G}$ and $J(\frac{z+\mathbf{x}^*}{2}) < \frac{J(z)+J(\mathbf{x}^*)}{2} = J(\mathbf{x}^*)$ since $\mathcal{G}$ is  convex and $J$ is strictly convex. However, this contradicts with the conclusion that $J(z)\geq J(\mathbf{x}^*)$ for all $z\in\mathcal{G}$. Hence $\mathbf{x}^*$ is the unique minimum of $J$ in the set $\mathcal{G}$. And the KKT condition is satisfied, i.e., $\nabla J(\mathbf{x}^*) +\sum_{i=1}^N \lambda_i d_i = 0$ for $d_i\in D\phi_i(\mathbf{x}^*)$. Then $\mathbf{x}^*$ is a weak local optimum. 

It is worth noting that if all $\phi_i$'s are smooth at $\mathbf{x}^*$, $\mathcal{G} = \mathcal{F}(\mathbf{x}^*)$. Then $\mathbf{x}^*$ is a fixed point, thus a strong local optimum by \cref{lemma: fixed pt}. 
\end{proof}

\begin{remark}
\cref{lem: strong descent} and \cref{lemma: cauchy} justify the adoption  of the terminate condition \cref{eq: terminal condition cost}, which is indeed equivalent to the standard terminate condition \cref{eq: terminal condition converge}, e.g.,  convergence in the objective function implies convergence in the solution. 
\end{remark}

\subsection{Proof of the Main Result}

\begin{proof} [Proof of \cref{thm: converge}]
If $\mathcal{F}^{(0)}$ is nonempty, then $\mathbf{x}^{(1)}\in \Gamma$ can be obtained by solving the convex optimization \cref{eq: cfs qp}. By \cref{lemma: existence}, $\mathcal{F}^{(1)}$ has nonempty interior, then $\mathbf{x}^{(2)}\in \Gamma$ can be obtained. By induction, we can conclude that $\mathbf{x}^{(i)}\in \mathcal{F}^{(i-1)}\subset\Gamma$ for $i=1,2,3,\cdots$.  Moreover, as a better solution is found at each iteration, then $J(\mathbf{x}^{(1)})\geq J(\mathbf{x}^{(2)})\geq \cdots$. This leads to two cases. The first case is that $J(\mathbf{x}^{(K)})= J(\mathbf{x}^{(K+1)})$ for some $K$, while the second case is that the cost keeps decreasing strictly, i.e.,  $J(\mathbf{x}^{(1)})> J(\mathbf{x}^{(2)})> \cdots$. 
In the first case, the condition $J(\mathbf{x}^{(K)})= J(\mathbf{x}^{(K+1)})$ is equivalent to $\mathbf{x}^{(K)}=\mathbf{x}^{(K+1)}$ by \cref{lem: strong descent}. By induction, the algorithm converges in the sense that $\mathbf{x}^{(k)}=\mathbf{x}^{(k+1)}$ and $J(\mathbf{x}^{(k)})= J(\mathbf{x}^{(k+1)})$ for all $k\geq K$. Moreover, as $\mathbf{x}^*:=\mathbf{x}^{(K)}$ is a fixed point, it is a strong local optimum by \cref{lemma: fixed pt}. 
If the cost keeps decreasing, e.g.,  $J(\mathbf{x}^{(1)})> J(\mathbf{x}^{(2)})> \cdots$, then the sequence $\{\mathbf{x}^{(k)}\}$ converges to a weak local optimum $\mathbf{x}^*$ by \cref{lemma: cauchy}.
\end{proof}

\section{Application on Motion Planning for Mobile Robots\label{sec: mobile robot}}
In this section, \cref{alg: cfs} is applied to a motion planning problem for mobile robots \cite{savkin2015safe}. Its application to other systems can be found in \cite{liu2017convex}. The problem will be formulated in \cref{sec: problem statement} and then transformed into the benchmark form \cref{eq: the problem} and \cref{eq: gamma}. The major difficulty in transforming the problem lies in finding the semi-convex function $\phi_i$ to describe the constraints. The method to construct $\phi_i$ in 2D is discussed in \cref{sec: transform the problem} and the resulting convex feasible set is illustrated through examples in \cref{sec: cfs example}. The performance of \cref{alg: cfs} is shown in \cref{sec: performance} and compared to interior point method (ITP) and sequential quadratic programming (SQP).

\subsection{Problem Statement\label{sec: problem statement}}
Suppose a mobile robot needs to plan a trajectory in $\mathbb{R}^2$ from the start point $x_0$ to the goal point $G$ as shown in \cref{fig: motion planning problem}. Let $x_q\in \mathbb{R}^2$ be the position of the robot at time step $q$. Define the decision variable as $\mathbf{x}:=\left[\begin{array}{ccc}x_1^T & \cdots & x_h^T\end{array}\right]^T\in \mathbb{R}^{2h}$ where $h$ is the planning horizon. The whole trajectory is denoted as $\bar{\mathbf{x}}:=\left[\begin{array}{ccc}x_0^T & \mathbf{x}^T & G^T\end{array}\right]^T\in \mathbb{R}^{2(h+2)}$. Define a sequence of selection functions $l_q:\mathbb{R}^{2h}\rightarrow\mathbb{R}^2$ as $x_q=l_q(\mathbf{x})$. The sampling time is  $t_s$. Let the velocity at $q$ be $v_q:=\frac{x_q-x_{q-1}}{t_s}$ and the acceleration at $q$ be $a_q:=\frac{v_q-v_{q-1}}{t_s}$. 

\begin{figure}[t]
\begin{center}
\includegraphics[width=7cm]{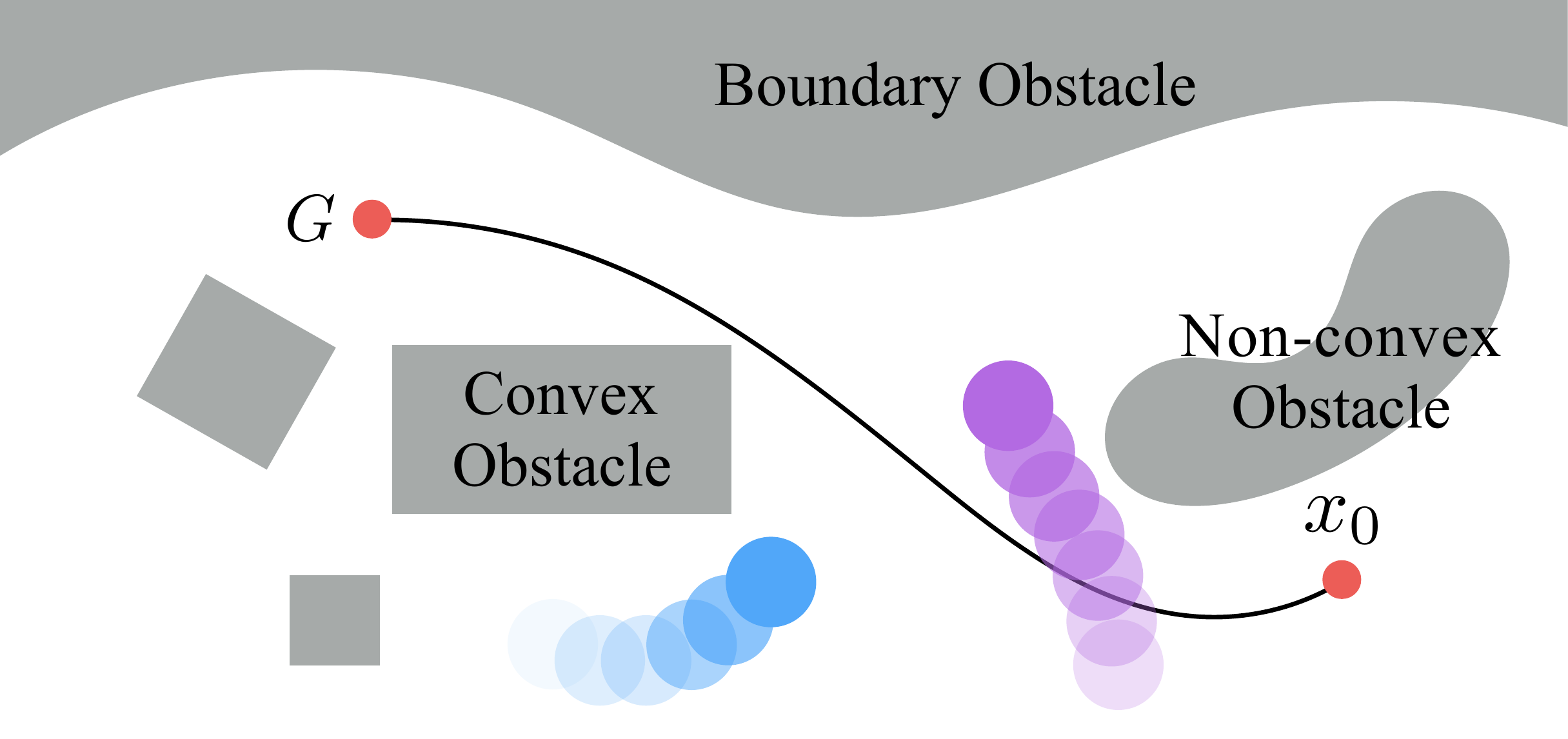}
\caption{The motion planning problem.}
\label{fig: motion planning problem}
\end{center}
\end{figure}

\subsubsection*{The Cost Function}
The cost function of the problem is designed as
\begin{equation}
J(\mathbf{x}) = w_1 \|\mathbf{x}-\mathbf{x}^r\|^2_Q+w_2 \|\mathbf{x}\|^2_S,\label{eq: cost planning}
\end{equation}
where $w_1,w_2\in\mathbb{R}^+$. The first term $ \|\mathbf{x}-\mathbf{x}^r\|^2_Q:=(\bar{\mathbf{x}}-\bar{\mathbf{x}}^r)^T Q (\bar{\mathbf{x}}-\bar{\mathbf{x}}^r)$ penalizes the distance from the target trajectory to the reference trajectory. The second term $\|\mathbf{x}\|^2_S:=\bar{\mathbf{x}}^T S\bar{\mathbf{x}}$ penalizes the properties of the target trajectory itself, e.g.,  length of the trajectory and magnitude of acceleration. 
The matrices $Q,S\in\mathbb{R}^{2(h+2)\times 2(h+2)}$ can be constructed from the following components:
1) matrix for position $Q_1:=I_{2(h+2)}$; 2) matrix for velocity $Q_2:=V^TV$ and 3) matrix for acceleration $Q_3:=A^TA$. Note that $V\in\mathbb{R}^{2(h+1)\times 2(h+2)}$ and $A\in\mathbb{R}^{2h\times 2(h+2)}$ are defined as
\begin{equation*}
V=\footnotesize{\frac{1}{t_s}\left[\begin{array}{ccccc}
I_2 & -I_2 & 0 & \cdots & 0\\
0 &  I_2 & -I_2 &\cdots & 0\\
\vdots & \vdots & \ddots & \ddots &\vdots\\
0 & 0 & \cdots & I_2 & -I_2
\end{array}\right]}, A=\frac{1}{t_s^2}\left[\begin{array}{cccccc}
I_2 & -2I_2 & I_2 & 0 & \cdots & 0\\
0 &  I_2 & -2I_2 & I_2 &\cdots & 0\\
\vdots & \vdots & \ddots & \ddots &\ddots &\vdots\\
0 & 0 &\cdots & I_2 & -2I_2 & I_2
\end{array}\right],
\end{equation*}
which take finite differences of the trajectory $\mathbf{x}$ such that $V\mathbf{x}$ returns the velocity vector and $A\mathbf{x}$ returns the acceleration vector. Then $Q:=\sum_{i=1}^3 c^q_iQ_i$ and $S:=\sum_{i=1}^3 c^s_i Q_i$ where $c^q_i$ and $c^s_i$ are positive constants. \cref{assump: cost} is satisfied.

\subsubsection*{The Constraints}
The obstacles in the environment are denoted as $\mathcal{O}_j\in\mathbb{R}^2$ for $j\in\mathbb{N}$. Each $\mathcal{O}_j$ is simply connected and open with piecewise smooth boundary that does not contain any sharp concave corner\footnote{Note that the obstacles may not be physical, but denote the infeasible area in the state space. A sharp concave corner in $\mathbb{R}^2$ is a concave corner with $360^{\circ}$ tangent angle. Existence of a sharp concave corner will violate \cref{assump: constraint}, since there does not exist a 2D convex neighborhood in $\mathcal{O}_j^c$ at a sharp concave corner.}.  For example, there are seven such obstacles in \cref{fig: motion planning problem} where five of them are static and two are dynamic. Let $\mathcal{O}_j$ denote obstacle $j$ when centered at the origin. The area occupied by obstacle $j$ at time step $q$ is defined as $\mathcal{O}_j(q)$. Define a linear isometry $\mathcal{T}_{j,q}: \mathbb{R}^2\rightarrow \mathbb{R}^2$ such that $\mathcal{O}_j=\mathcal{T}_{j,q}(\mathcal{O}_j(q))$. Then $x_q\notin\mathcal{O}_j(q)$ is equivalent to $\mathcal{T}_{j,q}(x_q)\notin\mathcal{O}_j$. 
Hence the constraint for the optimization is
\begin{equation}
\mathbf{x}\in \Gamma:=\{\mathbf{x}\in\mathbb{R}^{2h}: \mathcal{T}_{j,q}( l_q(\mathbf{x}))\notin \mathcal{O}_j,\forall j,\forall q=1,\cdots, h\}.\label{eq: 2d constraint}
\end{equation}
It is easy to verify that \cref{assump: constraint} is satisfied.

\subsection{Transforming the Problem\label{sec: transform the problem}}

In order to apply \cref{alg: cfs}, a semi-convex decomposition \cref{eq: gamma} satisfying \cref{assump: regularity} needs to be performed. For example, obstacles containing concave corners need to be partitioned into several overlapping obstacles that do not have concave corners as discussed in \cref{sec: analytical constraint}. Without loss of generality, $\mathcal{O}_j$ is assumed to represent obstacles after decomposition such that it is either a convex obstacle, or a boundary obstacle, or a non-convex obstacle as shown in \cref{fig: motion planning problem}. For each $\mathcal{O}_j$, we first construct a simple function $\varphi_j:\mathbb{R}^2\rightarrow\mathbb{R}$  and then use $\varphi_j$ to construct $\phi_i$. The function $\varphi_j$ is continuous, piecewise smooth and semi-convex such that $\mathcal{O}_j=\{x\in\mathbb{R}^2:\varphi_j(x)<0\}$ and $\partial\mathcal{O}_j=\{x\in\mathbb{R}^2:\varphi_j(x)=0\}$. We call $\varphi_j$ a safety index in the following discussion, since it typically measures the distance to an obstacle. The construction of $\varphi_j$ in each case is discussed below.

\subsubsection*{Convex obstacle}  A convex obstacle refers to the case that $\mathcal{O}_j$ is compact and convex. In this case, $\varphi_j$ is defined to be the signed distance function to $\mathcal{O}_j$, i.e., 
\begin{equation}
\varphi_j(x):=\left\{\begin{array}{ll}
\min_{y\in\partial\mathcal{O}_j} \|x-y\| & x\notin\mathcal{O}_j\\
-\min_{y\in\partial\mathcal{O}_j} \|x-y\| & x\in\mathcal{O}_j
\end{array}\right..\label{eq: safety index convex}
\end{equation}

\subsubsection*{Boundary obstacle} A boundary obstacle refers to a non-compact $\mathcal{O}_j$ such that there is an affine parameterization of $\partial \mathcal{O}_j$, e.g.,  if we rotate and align the obstacle properly, there exists a continuous and piecewise smooth semi-convex function $f:\mathbb{R}\rightarrow\mathbb{R}$ such that $p_2=f(p_1)$ describes $\partial \mathcal{O}_j$ where $x=(p_1,p_2)\in\mathbb{R}^2$. Then $\varphi_j$ is defined as the directional distance along $(0,1)$ to the boundary $\partial \mathcal{O}_j$, i.e.,
\begin{equation}
\varphi_j(x) := f(p_1)-p_2.\label{eq: safety index boundary}
\end{equation}

\subsubsection*{Non-convex obstacle} A non-convex obstacle refers to the case that $\mathcal{O}_j$ is compact, but non-convex. As the signed distance function for a non-convex set is not semi-convex, the strategy is to introduce directional distance. 
Denote the smallest convex envelop of the obstacle $\mathcal{O}_j$ as $\widehat{\mathcal{O}}_j$. Then $\varphi_j$ is defined as a signed directional distance function which computes the minimum distance from a point $x$ to $\partial\mathcal{O}_j$ in the direction that is perpendicular to $\partial\widehat{\mathcal{O}}_j$ as shown in \cref{fig: concave contour}. 
Let $\delta$ be a correspondence function that maps a point $y\in\partial\widehat{\mathcal{O}}_j$ to the closest $z\in \partial\mathcal{O}_j$ such that $z-y$ is perpendicular to $\partial\widehat{\mathcal{O}}_j$ at $y$. It is assumed that $\delta$ is bijective and $\|y-\delta(y)\|$ is semi-convex in $y$.
Then
\begin{equation}
\varphi_j(x) = \left\{\begin{array}{ll}
\min_{y\in\partial\widehat{\mathcal{O}}_j} [\|x-y\| + \|y-\delta(y)\|] & x\notin\widehat{\mathcal{O}}_j\\
-\min_{y\in\partial\widehat{\mathcal{O}}_j} [\|x-y\| - \|y-\delta(y)\|] & x\in\widehat{\mathcal{O}}_j
\end{array}\right..\label{eq: safety index non-convex}
\end{equation}

\begin{figure}[t]
\begin{center}
\subfloat[Convex obstacle.]{\includegraphics[width=3cm]{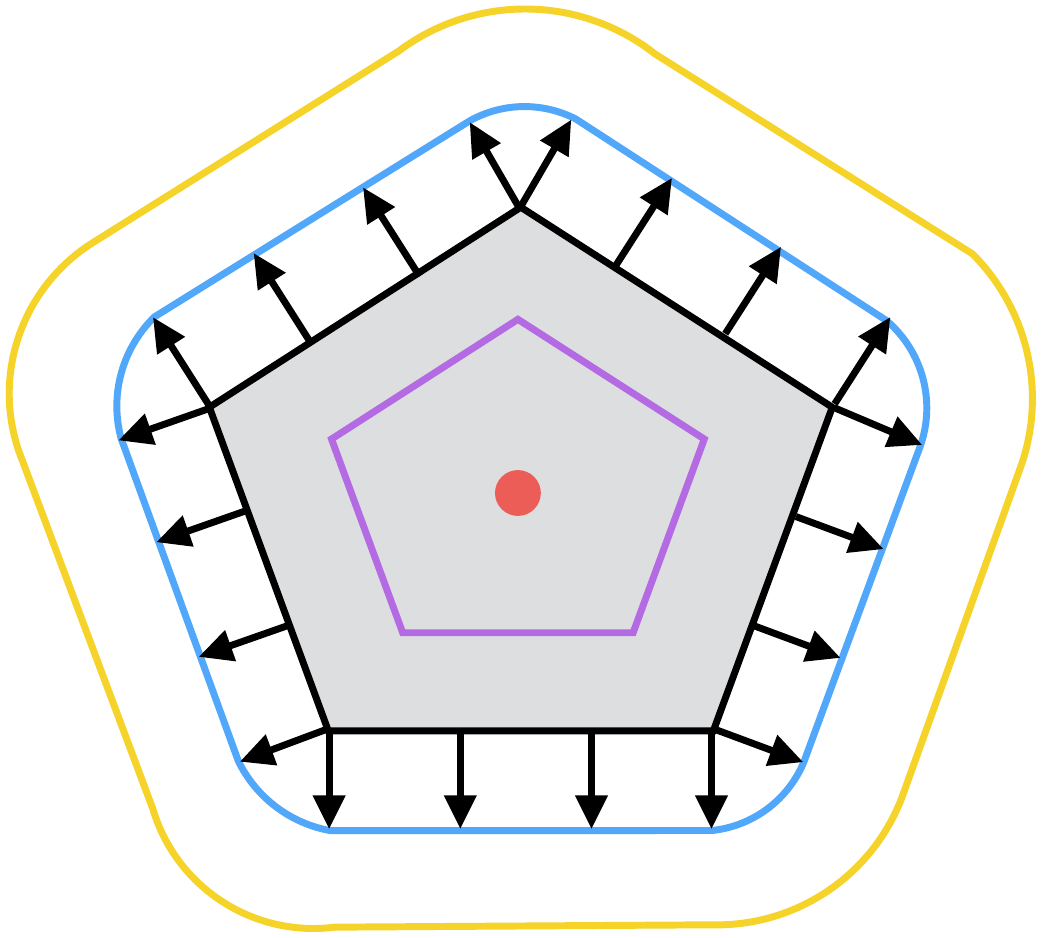}}~~
\subfloat[Concave obstacle.\label{fig: concave contour}]{\includegraphics[width=3.5cm]{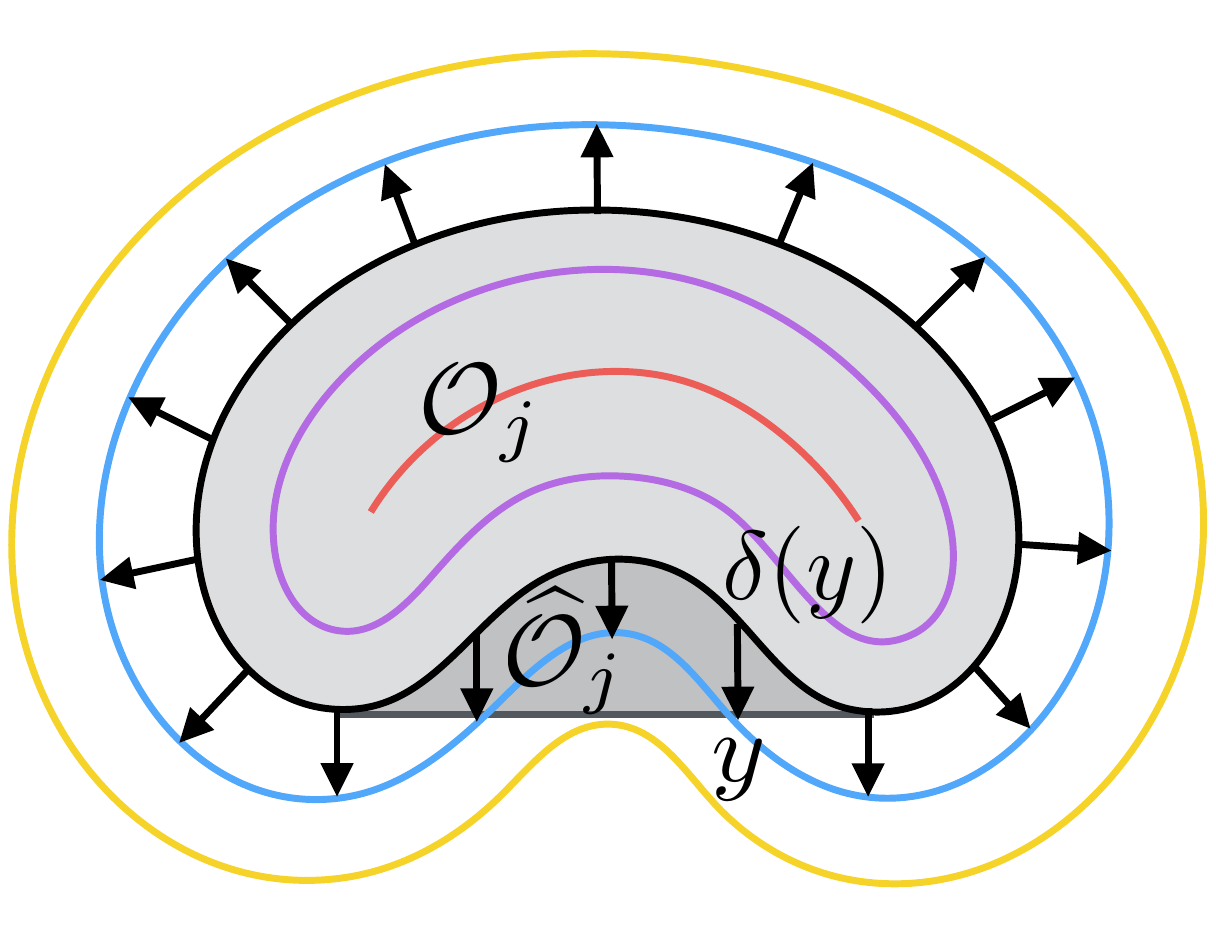}}~~
\subfloat[Boundary obstacle.]{\includegraphics[width=3cm]{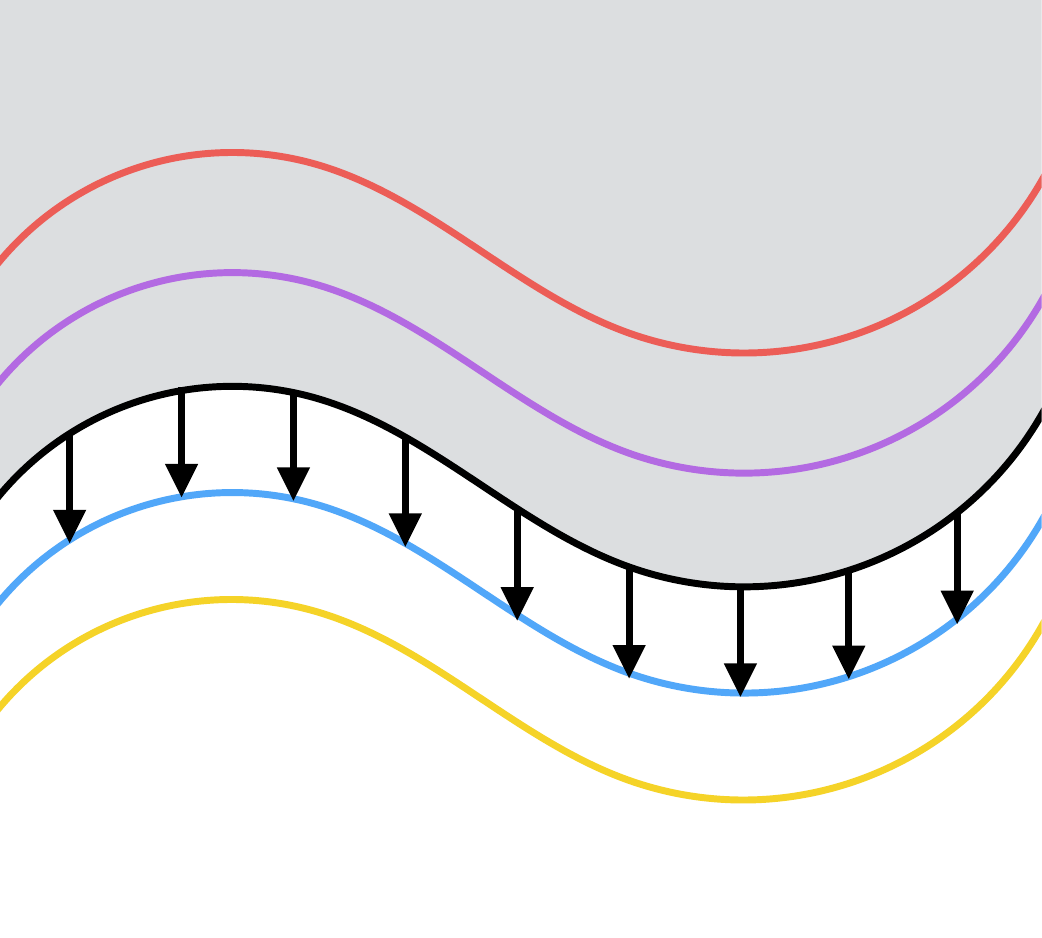}}
\caption{Contours of the safety indices for typical obstacles.}
\label{fig: obstacle contours}
\end{center}
\end{figure}

It is easy to verify that $\varphi_j$ in all three cases are continuous, piecewise smooth, semi-convex, and satisfy the first two arguments in \cref{assump: regularity}. Moreover, $\varphi_j$ is strictly convex for convex obstacles. According to \cref{eq: safety index convex}, \cref{eq: safety index boundary} and \cref{eq: safety index non-convex}, the contours of the safety indices are shown in \cref{fig: obstacle contours}. 
Based on the safety indices, \cref{eq: 2d constraint} can be re-written as
\begin{equation}
\mathbf{x}\in\Gamma = \bigcap_{j,q}\{\mathbf{x}\in\mathbb{R}^{2h}:\varphi_j(\mathcal{T}_{j,q}( l_q(\mathbf{x})))\geq 0\}.\label{eq: def phi}
\end{equation}
Define $\phi_{j,q}(\mathbf{x}):=\varphi_j(\mathcal{T}_{j,q}( l_q(\mathbf{x})))$ for all $j$ and $q$. Let $\Gamma_{j,q}:=\{\mathbf{x}\in\mathbb{R}^{2h}:\phi_{j,q}(\mathbf{x})\geq 0\}$. When the obstacles do not overlap with each other at every time step as shown in \cref{fig: motion planning problem}, the third argument in \cref{assump: regularity} is satisfied. 
Then $\{\phi_{j,q}\}_{j,q}$ is a decomposition of $\Gamma$ that satisfies \cref{assump: regularity}. 
The convex feasible set $\mathcal{F}:=\bigcap_{j,q}\mathcal{F}_{j,q}$ can be constructed according to the discussion in \cref{sec: find cfs}. And \cref{alg: cfs} can be applied.

\subsection{The Convex Feasible Set - Examples\label{sec: cfs example}} 
This section illustrates the configuration of convex feasible sets with examples. 
Since $\phi_{j,q}$ only depends on $x_q$, $\mathcal{F}_{j,q}(\mathbf{x}^r)$ only constraint $x_q$. For example, according to \cref{eq: def cfs convex}, the convex feasible set for a convex $\phi_{j,q}$ is 
\begin{equation}
\mathcal{F}_{j,q}(\mathbf{x}^r)=\{\mathbf{x}:\varphi_{j}(\mathcal{T}_{j,q}(x_q^r))+\hat \nabla \varphi_{j} (\mathcal{T}_{j,q}(x_q^r))\nabla\mathcal{T}_{j,q}(x_q^r)(x_q-x_q^r)\geq 0\}.\label{eq: reduced order}
\end{equation}
For simplicity, define $\mathcal{F}^q_{j,q}(\mathbf{x}^r):=l_q(\mathcal{F}_{j,q}(\mathbf{x}^r))\in\mathbb{R}^2$. In the following discussion, we will first illustrate $\mathcal{F}^q_{j,q}(\mathbf{x}^r)$ in $\mathbb{R}^2$ and then $\bigcap_q\mathcal{F}_{j,q}(\mathbf{x}^r)$ in $\mathbb{R}^{2h}$.

\subsubsection*{The convex feasible set in $\mathbb{R}^2$}
We illustrate the convex feasible set for one obstacle at one time step. For simplicity, subscripts $j$ and $q$ are removed and $\mathcal{T}_{j,q}$ is assumed to be identity.
Consider a polygon obstacle with vertices $a=(1,1)$, $b=(-1,1)$, $c=(-1,-1)$ and $d=(1,-1)$ as shown in \cref{fig: cfs2}. Let $\|x\|_\infty$ denote the $l_\infty$ norm, i.e., $\|x\|_\infty:=\max\{|p_1|,|p_2|\}$. Then according to \cref{eq: safety index convex}, 
\begin{equation}
\varphi(x) = \left\{\begin{array}{ll}
\max\{-1-p_1,p_1-1,p_2-1,-1-p_2\} & \|x\|_\infty\leq1\\
\min\{\|x-a\|,\|x-b\|,\|x-c\|,\|x-d\|\} & |p_1|>1, |p_2|>1\\
\min\{|p_1-1|,|p_1+1|\} & |p_1|>1,|p_2|<1\\
\min\{|p_2-1|,|p_2+1|\} & |p_1|<1,|p_2|>1
\end{array}\right..
\end{equation}

For a reference point $x^r=(-1.5,-1.5)$, $\varphi(x^r)=\frac{\sqrt{2}}{2}$ and $\nabla\varphi(x^r)=\frac{\sqrt{2}}{2}[-1,-1]$. The convex feasible set\footnote{When implementing \cref{alg: cfs} in software, there is no need to explicitly compute the safety indices as shown in the examples. The solver just needs to know the rules \cref{eq: safety index convex} \cref{eq: safety index boundary} and \cref{eq: safety index non-convex} in computing those indices. The gradients or hessians of the safety indices can be computed numerically.} is $\mathcal{F}(x^r) = \{x:\varphi(x^r)+\nabla\varphi(x^r)(x-x^r)\geq 0\} = \{x:\frac{\sqrt{2}}{2}(-p_1-p_2+2)\geq0\}$. The bowl-shaped surface in \cref{fig: cfs2} illustrates the safety index $\varphi$. The plane that is tangent to the safety index satisfies the function $\frac{\sqrt{2}}{2}(-p_1-p_2+2)=0$. Since $\varphi$ is convex, the tangent plane is always below $\varphi$. The convex feasible set can be regarded as the projection of the positive portion of the tangent plane onto the zero level set.

\begin{figure}[t]
\begin{center}
\subfloat[Convex obstacle.\label{fig: cfs2}]{\includegraphics[width=6cm]{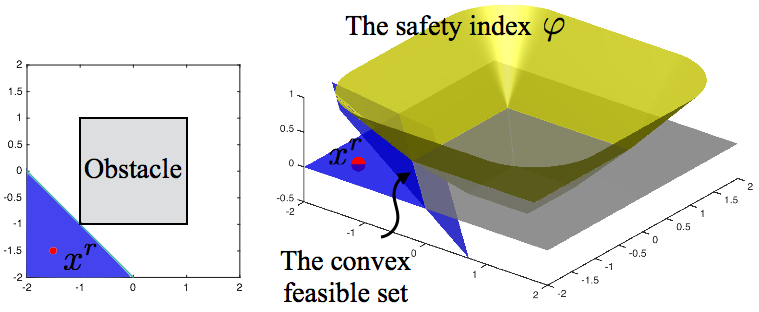}}~~
\subfloat[Non-convex obstacle.\label{fig: cfs1}]{\includegraphics[width=6cm]{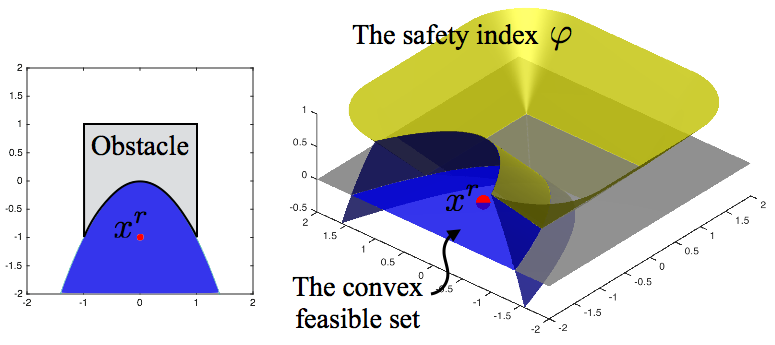}}
\caption{Illustration of the convex feasible set in $\mathbb{R}^2$.}
\end{center}
\end{figure}

Consider the case that points $c$ and $d$ are not connected by a straight line, but a concave curve $p_2=-(p_1)^2$ as shown in \cref{fig: cfs1}. Then according to \cref{eq: safety index non-convex}, 
\begin{equation}
\varphi(x) = \left\{\begin{array}{ll}
\max\{-(p_1)^2-p_2,p_1-1,p_2-1,-1-p_2\} & \|x\|_\infty\leq1, p_2\geq-(p_1)^2\\
\min\{\|x-a\|,\|x-b\|,\|x-c\|,\|x-d\|\} & |p_1|>1, |p_2|>1\\
\min\{|p_1-1|,|p_1+1|\} & |p_1|>1,|p_2|<1\\
p_2-1 & p_2>1\\
-(p_1)^2-p_2 & p_2<-(p_1)^2
\end{array}\right..
\end{equation}

The hessian of $\varphi$ is bounded below by $-H^*=[-2,0;0,0]$. For a reference point $x^r=(0,-1)$, $\varphi(x^r)=1$ and $\nabla\varphi(x^r)=[0,-1]$. The convex feasible set is $\mathcal{F}(x^r) = \{x:\varphi(x^r)+\nabla\varphi(x^r)(x-x^r)\geq \frac{1}{2}(x-x^r)H^*(x-x^r)\} = \{x:-p_2-(p_1)^2\geq 0\}$. The bowl-shaped surface in \cref{fig: cfs1} illustrates the safety index $\varphi$. The parabolic surface that is tangent to the safety index represents the function $-p_2-(p_1)^2=0$. The convex feasible set is the projection of the positive portion of the surface onto the zero level set.

\subsubsection*{The convex feasible set in higher dimension}
We illustrate the convex feasible set for one obstacle over the entire time horizon. The obstacle is shown in \cref{fig: time aug 2d}, which is a translated version of the obstacle in \cref{fig: cfs2}. The reference trajectory violates the obstacle avoidance constraint. \cref{fig: time aug 3d} shows the convex feasible sets computed for each time step. Those sets formulate a corridor  around the time-augmented obstacle. A new trajectory will be computed in the corridor. Although the projection of the corridor into $\mathbb{R}^2$ is not convex, each time slice of the corridor is convex. In $\mathbb{R}^{2h}$, those slices are sticked together orthogonally, hence formulate a convex subset of $\mathbb{R}^{2h}$.

\begin{figure}[t]
\begin{center}
\subfloat[The reference trajectory.\label{fig: time aug 2d}]{\includegraphics[width=3.7cm]{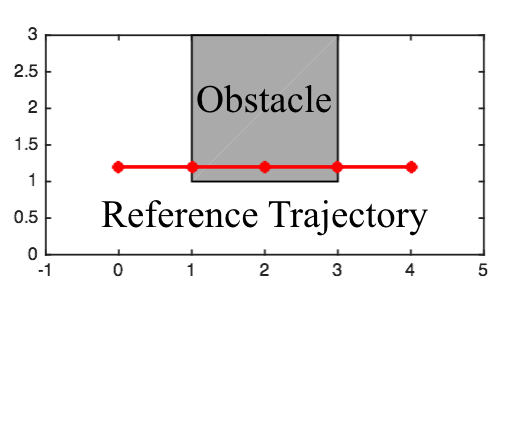}}~~
\subfloat[The convex feasible set over time.\label{fig: time aug 3d}]{\includegraphics[width=5cm]{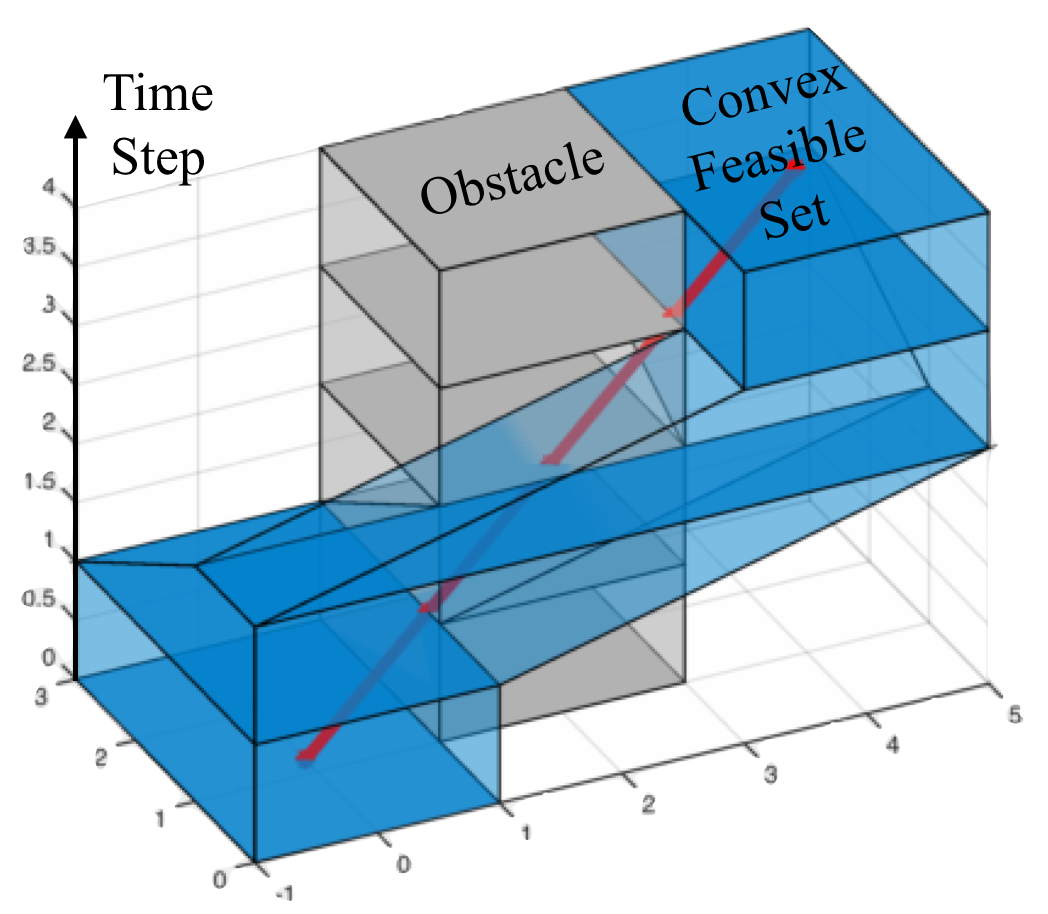}}
\caption{Illustration of the convex feasible set in $\mathbb{R}^{2h}$.}
\label{fig: time aug}
\end{center}
\end{figure}

As pointed out in \cref{remark: existence}, for an infeasible reference trajectory, when there are multiple obstacles, the existence of a corridor that bypasses all time-augmented obstacles under the proposed algorithm is hard to guarantee. In \cref{lemma: feasibility new}, we show that a convex feasible set has nonempty interior for any reference trajectory if certain geometric conditions are satisfied.

\begin{lemma} [Feasibility] \label{lemma: feasibility new}
If all obstacles are convex and have disjoint closures, i.e., $\bar{\mathcal{O}}_i(q)\bigcap\bar{\mathcal{O}}_j(q)=\emptyset$ for all $q$ and $i\neq j$, then the convex feasible set $\mathcal{F}(\mathbf{x}^r)$ has nonempty interior for all $\mathbf{x}^r\in\mathbb{R}^n$ when $\phi_{j,q}$ is chosen according to \cref{eq: safety index convex} and \cref{eq: def phi}.\label{lem: convex feasibility}
\end{lemma}
\begin{proof}
Considering \cref{lemma: existence}, we only need to prove that $\mathcal{F}(\mathbf{x}^r)$ has nonempty interior for infeasible $\mathbf{x}^r$. Infeasibility of $\mathbf{x}^r$ implies that some $\phi_{j,q}(\mathbf{x}^r)$ is negative. Fix $q$, since $\bar{\mathcal{O}}_j(q)$'s are disjoint, only one $\phi_{j,q}$ can be negative. Without loss of generality, suppose $\phi_{1,q}(\mathbf{x}^r)<0$ and $\phi_{j,q}(\mathbf{x}^r)\geq 0$ for $j\geq 2$. Since the safety index $\varphi_j$ is a signed distance function in \cref{eq: safety index convex} and $\mathcal{T}_{j,q}$ is an isometry, $\|\hat\nabla\varphi_j\| = 1$ and $\|\mathcal{T}_{j,q}\| = 1$. Then the ball $B(x_q^r, \varphi_{j}(\mathcal{T}_{j,q}(x_q^r)))$ is a subset of $\mathcal{F}_{j,q}^q(\mathbf{x}^r)$ according to \cref{eq: reduced order} for all $j$ and $q$. Hence $B(x_q^r, d^*)\subset \bigcap_{j\geq 2}\mathcal{F}_{j,q}^q(\mathbf{x}^r)$ where $d^* = \min_{j\geq 2}\varphi_{j}(\mathcal{T}_{j,q}(x_q^r))$ is the minimum distance to $\bigcup_{j\geq 2}\mathcal{O}_j(q)$. 
According to \cref{eq: safety index convex}, $\mathcal{F}^q_{1,q}(\mathbf{x}^r)$ is tangent to $\partial\mathcal{O}_1(q)$ at a point $x^*$ such that $\varphi_{1}(\mathcal{T}_{1,q}(x_q^r))= -\|x_q^r-x^*\|$ as shown in \cref{fig: feasibility convex}. Since $\bar{\mathcal{O}}_1(q)\bigcap\bar{\mathcal{O}}_j(q)=\emptyset$ for $j\geq 2$, then $d^* > -\varphi_{1}(\mathcal{T}_{1,q}(x_q^r)) = \|x^r_q-x^*\|$, which implies that the set $\mathcal{F}_{1,q}^q(\mathbf{x}^r)\bigcap B(x_q^r, d^*)$ has nonempty interior. Then $\bigcap_j \mathcal{F}^q_{j,q}(\mathbf{x}^r)$ has nonempty interior. So $\mathcal{F}(\mathbf{x}^r)=\oplus_q(\bigcap_j \mathcal{F}^q_{j,q}(\mathbf{x}^r))$ has nonempty interior where $\oplus$ means direct sum. 
\end{proof}
\begin{figure}[t]
\begin{center}
\includegraphics[width=4cm]{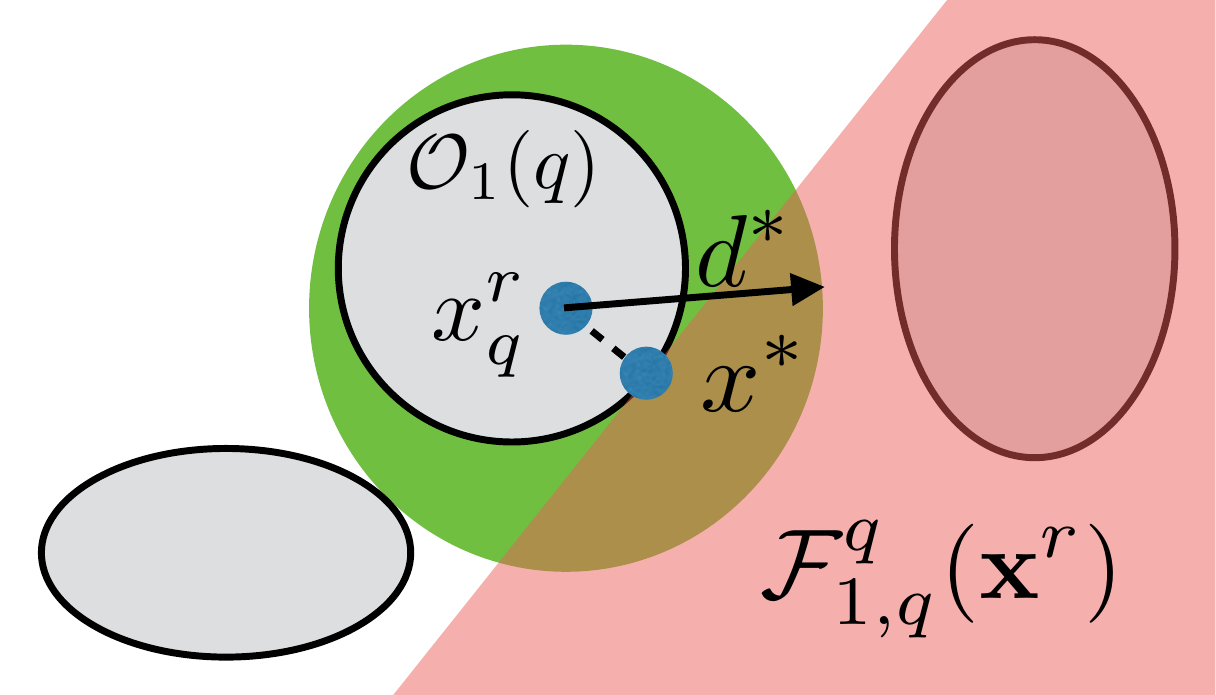}
\caption{Existence of convex feasible set for an infeasible reference point.}
\label{fig: feasibility convex}
\end{center}
\end{figure}

\begin{remark}
\cref{lem: convex feasibility} implies that $\mathcal{F}^{(0)}$ is nonempty for any $\mathbf{x}^{(0)}$.  By \cref{thm: converge}, the algorithm converges to a local optimum for any $\mathbf{x}^{(0)}\in\mathbb{R}^n$ if all obstacles  are convex and have disjoint closures.
\end{remark}

\subsection{Performance and Comparison\label{sec: performance}}

The performance of \cref{alg: cfs} will be illustrated through two examples, which will also be compared to the performance of existing non-convex optimization methods, interior point (ITP) and sequential quadratic programming (SQP). For simplicity, only convex obstacles and convex boundaries are considered\footnote{The non-convex obstacles or boundaries can either be partitioned into several convex components or be replaced with their convex envelops. Moreover, in practice, obstacles are measured by point clouds. The geometric information is extracted by taking convex hull of the points. Hence it automatically partitions the obstacles into several convex polytopes.}. \cref{alg: cfs} is implemented in both Matlab and C++. The convex optimization problem \cref{eq: cfs qp} is solved using the interior-point-convex method in \texttt{quadprog} in Matlab and the interior point method in Knitro \cite{Byrd2006} in C++. For comparison, \cref{eq: the problem} is also solved directly using ITP and SQP methods in \texttt{fmincon} \cite{centeroptimization} in Matlab and in Knitro in C++. To create fair comparison, the gradient and the hessian of the objective function $J$ and the optimal sub-gradients $\hat\nabla\phi_i$ of the constraint function $\phi_i$'s are also provided to ITP and SQP solvers.

In the examples, $x_0 = (0,0)$ and $G = (9,0)$. The planning horizon $h$ goes from $30$ to $100$. $t_s=(h+1)^{-1}$. The cost function \cref{eq: cost planning} penalizes the average acceleration along the trajectory, e.g.,  $Q=0$ and $S = h^{-1} A^TA$. When $h\rightarrow \infty$,  $J(\mathbf{x}) \rightarrow \int_0^1 \|\ddot{\mathbf{x}}\|^2 dt$ where $\mathbf{x}:[0,1]\rightarrow \mathbb{R}^n$ is a continuous trajectory with $\mathbf{x}(0) =x_0$ and $\mathbf{x}(1) = G$.
The initial reference $\mathbf{x}^{(0)}$ is chosen to be a straight line connecting $x_0$ and $G$ with equally sampled waypoints. In the first scenario, there are three disjoint convex obstacles as shown in \cref{fig: case1}. In the second scenario, there are three disjoint infeasible sets, two of which contain concave corners. Then they are partitioned into five overlapping convex obstacles as shown in \cref{fig: sim scenario}. In the constraint, a distance margin of $0.25$ to the obstacles is required. 

The computation times under different solvers are listed in \cref{table: comparison1}. The first column shows the horizon. In the first row, $h=100$ in scenario 1 and $h=60$ in scenario 2\footnote{In the C++ solver, the constraint is limited by $300$. Hence when there are five obstacles, the maximum allowed $h$ is $60$.}. In the remaining rows, $h$ is the same in the two scenarios. In the second column, ``-M'' means the algorithm is run in Matlab and ``-C'' means the algorithm is run in C++. Under each scenario, the first column shows the final cost. The second column is the total number of iterations. The third and fourth columns are the total computation time and the average computation time per iteration respectively (only the entries that are less than $100ms$ are shown). It does happen that the algorithms find different local optima, though CFS-M and CFS-C always find the same solution. In terms of computation time, \cref{alg: cfs} always outperforms ITP and SQP, since it requires less time per iteration and fewer iterations to converge. This is due to the fact that CFS does not require additional line search after solving \cref{eq: cfs qp} as is needed in ITP and SQP, hence saving time during each iteration. CFS requires fewer iterations to converge since it can take unconstrained step length $\|\mathbf{x}^{(k+1)}-\mathbf{x}^{(k)}\|$ in the convex feasible set as will be shown later. Moreover, \cref{alg: cfs} scales much better than ITP and SQP, as the computation time and time per iteration in CFS-C go up almost linearly with respect to $h$ (or the number of variables). 

The computation time of CFS consists of two parts: 1) the processing time, i.e., the time to compute $\mathcal{F}$ and 2) the optimization time, i.e., the time to solve \cref{eq: cfs qp}. As shown in \cref{fig: time decompose}, the two parts grow with $h$. In Matlab, the processing time dominates, while the optimization time dominates in C++. 

To better illustrate the advantage of \cref{alg: cfs}, the runtime statistics across all methods when $h=100$ in the first scenario are shown in \cref{fig: run time stats}. The first log-log figure shows the cost $J(\mathbf{x}^{(k)})$ versus iteration $k$, while the second semi-log figure shows the feasibility error $\max\{0,-\phi_1(\mathbf{x}^{(k)}),\cdots,-\phi_N(\mathbf{x}^{(k)})\}$ versus $k$. At the beginning, the cost $J=0$ and the feasibility error is $0.75$. In CFS-C and CFS-M, $\mathbf{x}^{(k)}$ becomes feasible at the first iteration while the cost $J(\mathbf{x}^{(k)})$ jumps up. In the following iterations, the cost goes down and converges to the optimal value. In ITP-C, the problem becomes feasible at the third iteration. In SQP-C, it becomes feasible at the fifth iteration. In order to make the problem feasible, the cost jumps much higher in ITP-C than in CFS-C. Once the problem is feasible, it also takes more iterations for ITP-C and SQP-C to converge compared to CFS-C. On the other hand, ITP-M and SQP-M have very small step length in the beginning. The problem only becomes feasible after 100 iterations. But once the problem is feasible, the performance of ITP-M and SQP-M is similar to that of ITP-C and SQP-C. Note that the cost below $1$ is not shown in the figure. 

The optimal trajectories computed by \cref{alg: cfs} for different $h$ in the first scenario is shown in \cref{fig: case1}. Those trajectories converge to a continuous trajectory when $h$ goes up. 
\cref{fig: sim scenario} illustrates the trajectories before convergence in CFS-C and ITP-C in scenario 2 when $h=50$. For ITP-C, the trajectories are shown every iteration in the first ten iterations and then every ten iterations in the remaining iterations. The step length in CFS-C is much larger than that in ITP-C, which explains why CFS requires fewer iterations to become feasible and fewer iterations to converge. The trajectories are feasible and smooth in every iteration in CFS. Hence in case of emergencies, we can safely stop the iterations and get a good enough feasible trajectory before convergence.

With respect to the results, we conclude that \cref{alg: cfs} is time-efficient, local-optimal and scalable.

\begin{figure}[t]
\begin{center}
\includegraphics[width=7cm]{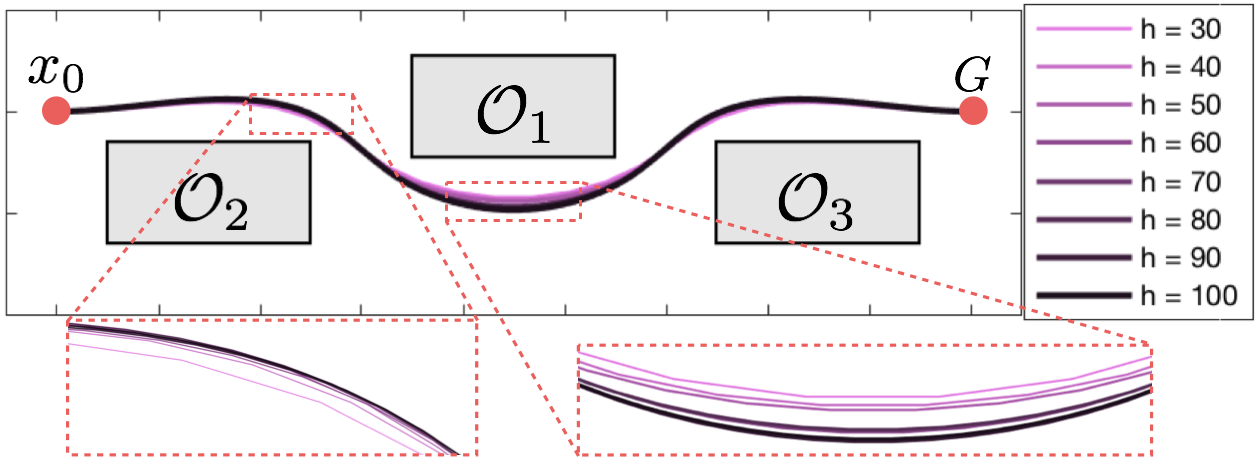}
\caption{Scenario 1 and the optimal trajectories for different horizon $h$.}
\label{fig: case1}
\end{center}
\end{figure}

\begin{table}[t]\scriptsize
\caption{Comparison among different algorithms.}
\begin{center}
\begin{tabular}{|c|c|c|c|c|c|c|c|c|c|}
\hline
\multicolumn{2}{|c|}{}&\multicolumn{4}{|c|}{Scenario 1} & \multicolumn{4}{|c|}{Scenario 2}\\
\hline
$h$ & Method &  Cost  & Iter & Time & dT &  Cost  & Iter & Time & dT\\
\hline
\multirow{6}{0.6cm}{100 or 60} & SQP-M & 1358.9 & 239 & 140.5s & - & 5385.5 & 198 & 60.1s & -\\
& ITP-M & 1358.9 & 470 & 50.6s & - & 5349.6 & 320 & 18.0s & 56.3ms\\
& CFS-M & 1358.9 & 18 & 1.8s & 98.8ms & 5413.2 & 6 & 344.9ms & 57.5ms\\
& SQP-C & 1347.6 & 123 & 47.3s & - & 5489.6 & 52 & 10.2s & -\\
& ITP-C & 1341.7 & 306 & 2.9s & 9.5ms & 5349.6 & 123 & 595.0ms & 4.8ms\\
& CFS-C & 1358.9 & 18 & \textbf{74.4ms} & \textbf{4.1ms} & 5413.2 & 6 & \textbf{27.3ms} & \textbf{4.6ms}\\
\hline
\multirow{6}{*}{50}  &SQP-M &  2299.5 & 110 & 25.8s & -& 5539.6 & 164 & 40.0s &  - \\
 &ITP-M  & 1308.3 & 187 & 8.8s & 47.1ms& 5372.8 & 268 & 11.3s & 42.2ms\\
 &CFS-M  & 1458.2 & 8 & 212.1ms & 26.5ms & 5394.2 & 5 & 186.1s & 37.2ms\\
 &SQP-C & 1308.3 & 52 & 3.4s & 65.4ms & 5682.0 & 62 & 5.7s & 91.9ms\\
 &ITP-C & 1275.1& 131 & 390ms & 3.0ms & 5555.8 & 96 & 495.6ms & 5.2ms\\
 &CFS-C  & 1458.2 & 8 & \textbf{23.7ms} & \textbf{3.0ms} &  5394.2 & 5 & \textbf{21.0ms} & \textbf{4.2ms} \\
\hline
\multirow{6}{*}{40}  & SQP-M &  3391.6 & 97 & 15.6s & - & 5318.1 & 127 & 23.2s & -\\
 & ITP-M & 1317.0& 150 & 5.2s & 34.7ms & 5549.8 & 156 & 5.6s & 35.9ms\\
 & CFS-M & 1317.0 & 8 & 172.6ms & 21.6ms & 5399.2 & 6 & 171.9 & 28.7ms \\
 & SQP-C & 1317.0 & 40 & 1.5s & 37.5ms & 5568.8 & 69 & 2.8s & 40.6ms\\
 & ITP-C	& 1170.5 & 102 & 240.5ms& 2.4ms & 5399.2 & 91 & 290.4ms & 3.2ms\\
 & CFS-C & 1317.0&8 & \textbf{16.5ms} & \textbf{2.1ms}& 5399.2 & 6 & \textbf{19.0ms} & \textbf{3.2ms}\\
\hline
\multirow{6}{*}{30}  & SQP-M & 1039.2 & 106 & 8.4s &79.2ms & 5075.8 & 90 & 11.0s & - \\
  & ITP-M & 1039.2  & 109 & 2.8s & 25.7ms & 5162.4 & 127 & 3.2s & 25.2ms\\
 & CFS-M & 1039.2  & 12 & 208.2ms & 17.3ms & 5167.3 & 5 & 110.6ms &22.1ms\\
 & SQP-C & 1453.3 & 27 & 379.1ms &14.0ms & 5444.3 & 42 & 1.5s & 35.7ms \\
 & ITP-C & 1039.2 & 59 &118.5ms& 2.0ms & 5320.8 & 67 & 125.5ms & 1.9ms	\\
 & CFS-C & 1039.2 & 12 & \textbf{19.0ms} & \textbf{1.6ms} & 5167.3 & 5 & \textbf{12.2ms} & \textbf{2.4ms}\\
\hline
\end{tabular}
\end{center}
\label{table: comparison1}
\end{table}%

\begin{figure}[t]
\begin{center}
\subfloat[Trajectories in CFS-C.]{
\includegraphics[width = 6.1cm]{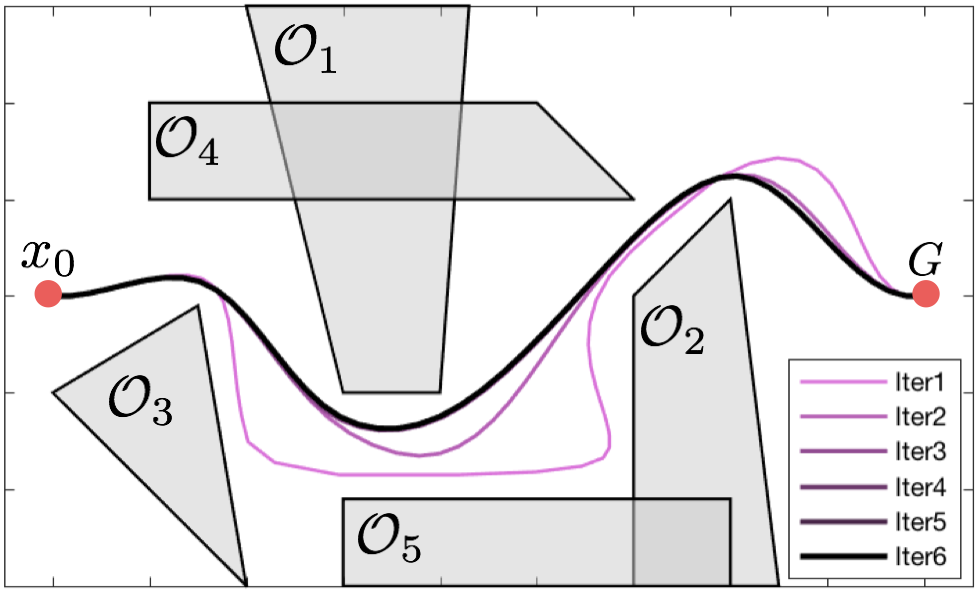}}
\subfloat[Trajectories in ITP-C.]{
\includegraphics[width = 6cm]{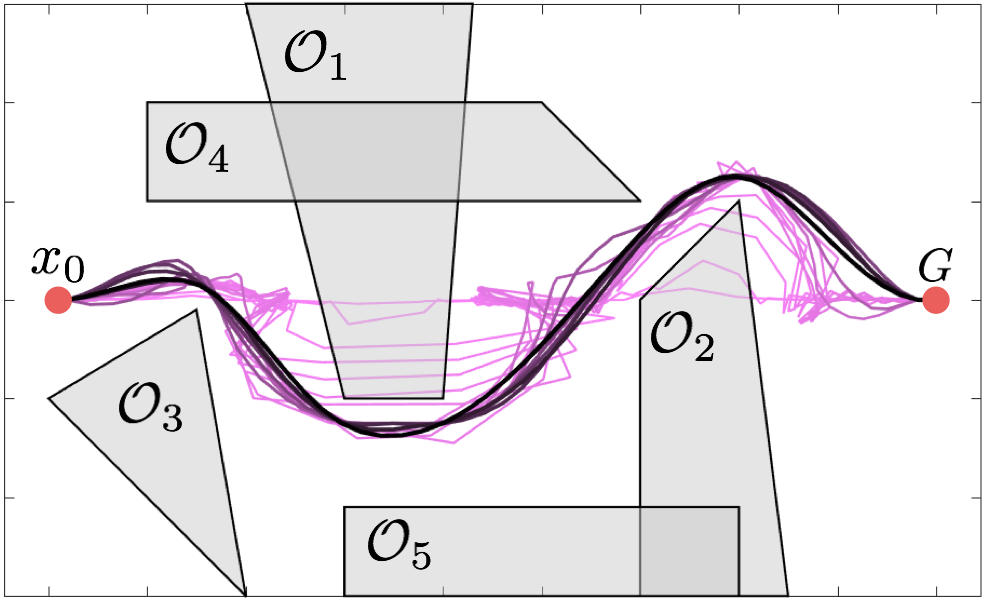}}
\caption{Scenario 2 and the trajectories before convergence for $h=50$.}
\label{fig: sim scenario}
\end{center}
\end{figure}

\begin{figure}[t]
\begin{center}
\includegraphics[width=13cm]{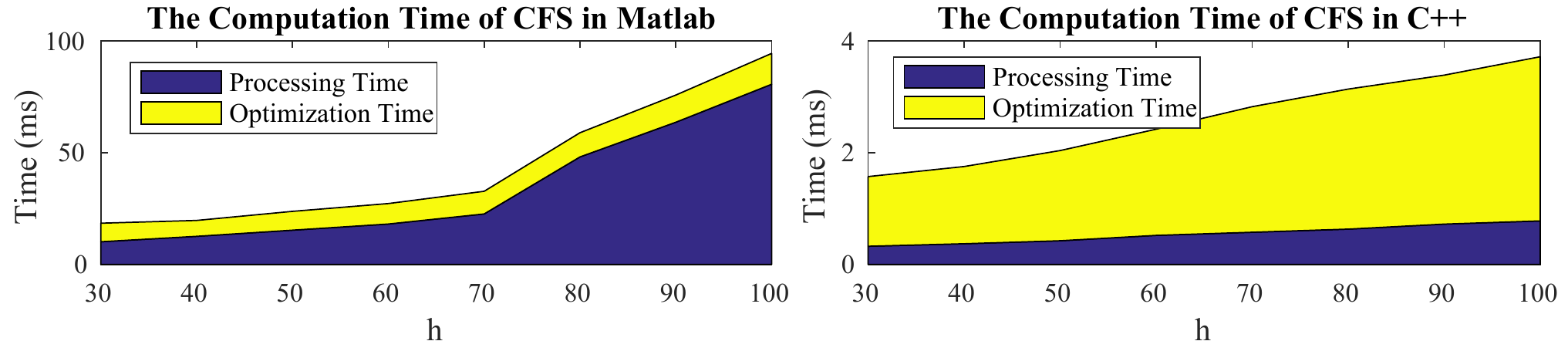}
\caption{The decomposed time per iteration using \cref{alg: cfs} in scenario 1.}
\label{fig: time decompose}
\end{center}
\end{figure}

\begin{figure}[t]
\begin{center}
\includegraphics[width=12cm]{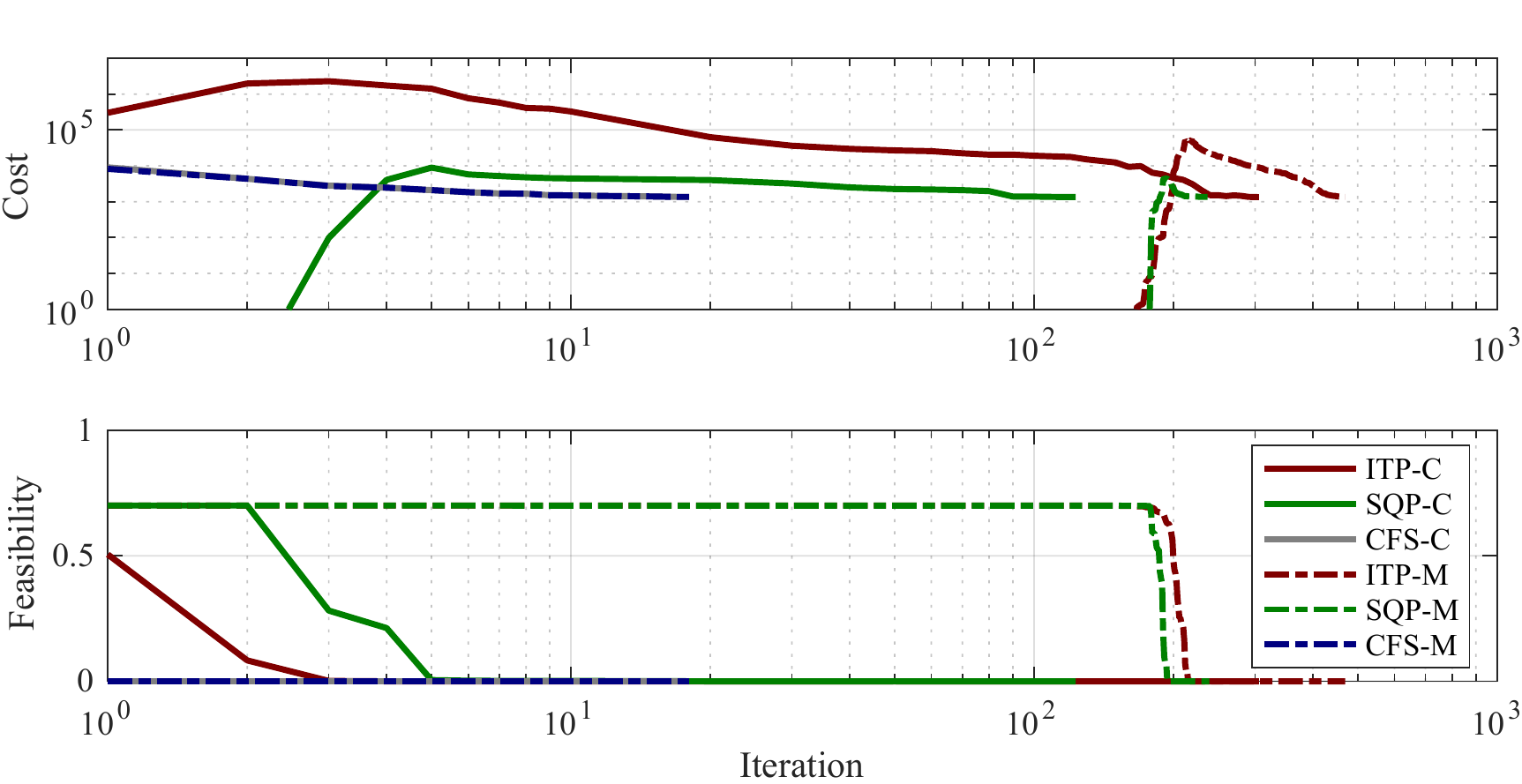}
\caption{The run time statistics in scenario 1.}
\label{fig: run time stats}
\end{center}
\end{figure}

\section{Conclusion\label{sec: conclusion}}
This paper introduced a fast algorithm for real time motion planning based on the convex feasible set. The CFS algorithm can handle problems that have convex cost function and non-convex constraints which are usually encountered in robot motion planning. By computing a convex feasible set within the non-convex constraints, the non-convex optimization problem is transformed into a convex optimization. Then by iteration, we can efficiently eliminate the error introduced by the convexification. It is proved in the paper that the proposed algorithm is feasible and stable. Moreover, it can converge to a local optimum if either the initial reference satisfies certain conditions or the constraints satisfy certain conditions. The performance of CFS is compared to that of ITP and SQP. It is shown that CFS reaches local optima faster than ITP and SQP, hence better suited for real time applications. In the future, methods for computing the convex feasible sets for infeasible references in complicated environments will be explored.

\bibliographystyle{siamplain}

\end{document}